\documentclass{amsart}

\usepackage{amssymb}
\newcommand{\cz}[2]{C_{#1}(#2)}
\newcommand{\cstar}[2]{C^{*}_{#1}(#2)}
\newcommand{\n}[2]{N_{#1}(#2)}

\newcommand{\op}[2]{O_{#1}(#2)}

\newcommand{\oupper}[2]{O^{#1}(#2)}

\newcommand{\ff}[1]{F(#1)}
\newcommand{\compp}[1]{\operatorname{comp}(#1)}
\newcommand{\comp}[2]{\operatorname{comp}_{#1}(#2)}
\newcommand{\compsol}[1]{\comp{\mathrm{sol}}{#1}}
\newcommand{\compasol}[1]{\comp{A,\mathrm{sol}}{#1}}
\newcommand{\layerr}[1]{E(#1)}
\newcommand{\layer}[2]{E_{#1}(#2)}
\newcommand{\layersol}[1]{\layer{\mathrm{sol}}{#1}}
\newcommand{\gfitt}[1]{F^{*}(#1)}
\newcommand{\sol}[1]{\operatorname{sol}(#1)}

\newcommand{\zz}[1]{Z(#1)}

\newcommand{\card}[1]{|\,#1\,|}
\newcommand{\primes}[1]{\pi(#1)}
\newcommand{\syl}[2]{\operatorname{Syl}_{#1}(#2)}
\newcommand{\aut}[1]{\operatorname{Aut}(#1)}
\newcommand{\out}[1]{\operatorname{Out}(#1)}

\newcommand{\gen}[2]{\langle \;#1 \mid #2\; \rangle}
\newcommand{\listgen}[1]{\langle \;#1\; \rangle}
\newcommand{\hyp}[1]{\operatorname{\rm Hyp}(#1)}

\newcommand{\set}[2]{\{ \;#1 \mid #2\;\}}
\newcommand{\listset}[1]{\{ \,#1\, \}}

\newcommand{\stab}[2]{\operatorname{Stab}_{#1}(#2)}

\newcommand{\field}[1]{{\mathbb #1}}
\newcommand{\fieldchar}[1]{{\operatorname{char}}\,\field{#1}}
\newcommand{\gf}[1]{\operatorname{GF}(#1)}

\newcommand{\pring}[2]{\field{#1}[#2]}
\newcommand{\gfpring}[2]{\gf{#1}[#2]}

\newcommand{\gl}[2]{\operatorname{GL}_{#1}(#2)}
\newcommand{\ltwo}[1]{\operatorname{L}_{2}(#1)}
\newcommand{\sltwo}[1]{\operatorname{SL}_{2}(#1)}
\newcommand{\uthree}[1]{\operatorname{U}_{3}(#1)}
\newcommand{\sz}[1]{\operatorname{Sz}(#1)}
\newcommand{\spfour}[1]{\operatorname{Sp}_{4}(#1)}
\newcommand{\gtwo}[1]{G_{2}(#1)}
\newcommand{\twistedgtwo}[1]{{^{2}G_{2}(#1)}}
\newcommand{\twistedffour}[1]{{^{2}F_{4}(#1)}}
\newcommand{\chevgroup}[2]{{^{#1}\mathcal L(#2)}}
\newcommand{\chev}[1]{\operatorname{chev}(#1)}
\newcommand{\alt}[1]{\operatorname{Alt}(#1)}
\newcommand{\sym}[1]{\operatorname{Sym}(#1)}
\newcommand{\badfour}{\listset{\ltwo{2^{r}}, \ltwo{3^{r}}, \uthree{2^{r}}, \sz{2^{r}}}}

\newcommand{\kernel}[1]{\ker #1}
\newcommand{\kernelon}[2]{\ker(#1\; \mbox{on}\; #2)}

\newcommand{\isomorphic}{\cong}
\newcommand{\normal}{\unlhd\,}
\newcommand{\subnormal}{\unlhd\unlhd\,}
\newcommand{\characteristic}{\operatorname{char}}

\newcommand{\br}[1]{\overline{#1}}
\newcommand{\nonid}{^\#}

\newtheorem{Theorem}{Theorem}[section]
\newtheorem*{UnnumberedTheorem}{Theorem}
\newtheorem{Lemma}[Theorem]{Lemma}
\newtheorem{Corollary}[Theorem]{Corollary}
\newtheorem{Hypothesis}[Theorem]{Hypothesis}

\newtheorem{Claim}{Claim}

\theoremstyle{definition}
\newtheorem{Definition}[Theorem]{Definition}

\theoremstyle{remark}

\newtheorem*{UnnumberedRemark}{Remark}

\numberwithin{equation}{section}

\begin{document}

\title{Automorphisms of $K$-groups I}


\author{Paul Flavell}
\address{The School of Mathematics\\University of Birmingham\\Birmingham B15 2TT\\Great Britain}
\email{P.J.Flavell@bham.ac.uk}
\thanks{A considerable portion of this research was done whilst the author was in receipt
of a Leverhulme Research Project grant and during visits to the Mathematisches Seminar,
Christian-Albrechts-Universit\"{a}t, Kiel, Germany.
The author expresses his thanks to the Leverhulme Trust for their support and
to the Mathematisches Seminar for its hospitality.}


\date{}

\begin{abstract}
    This is the first in a sequence of papers that will develop the
    theory of automorphisms of nonsolvable finite groups.
    The sequence will culminate in a new proof of McBride's
    Nonsolvable Signalizer Functor Theorem,
    which is one of the fundamental results required for the
    proof of the Classification of the Finite Simple Groups.
\end{abstract}

\maketitle
\section{Introduction}\label{intro}
The theory of automorphisms of finite solvable groups is very well developed.
A high point of that theory is Glauberman's
Solvable Signalizer Functor Theorem \cite{G}.
This is the first in a sequence of papers that will develop the
theory of automorphisms of arbitrary finite groups and will
culminate in a new proof of McBride's Nonsolvable Signalizer Functor Theorem
\cite{McB1, McB2}.
This proof will differ significantly from McBride's.
It will be modelled on the author's proof of the
Solvable Signalizer Functor Theorem \cite{F2}.

The Signalizer Functor Theorems played a crucial role in the
first generation proof of the Classification of the Finite Simple Groups.
They are also background results needed for the new proof of the
Classification in the Gorenstein-Lyons-Solomon book series \cite{GLS}.

It is not however the sole aim of this sequence of papers to prove
the Nonsolvable Signalizer Functor Theorem.
Many ideas are explored in much greater depth than is required for
that purpose and a more general theory ensues.
Consequently the results proved will be applicable in situations
where Signalizer Functor Theory is not.
Once this sequence of papers is complete,
it is the intention to prepare a monograph whose main focus will
be a proof of the Nonsolvable Signalizer Functor Theorem.

The results of this paper require the so-called $K$-group hypothesis.
Recall that a $K$-group is a finite group all of whose simple sections
are isomorphic to a cyclic group, an alternating group,
a group of Lie type or one of the $26$ sporadic simple groups.
The Classification asserts that every finite group is a $K$-group.
Thus, given the Classification, the $K$-group hypothesis is superfluous.
The main application of the Nonsolvable Signalizer Functor Theorem
is to analyze a minimal counterexample to the Classification.
In such a group, all proper subgroups are $K$-groups whence the
$K$-group hypothesis causes no difficulty.
In \S\ref{k} we will state explicitly the properties of simple
$K$-groups that we use.

Let $A$ be a group that acts as a group of automorphisms on the group $G$.
Assume that $A$ and $G$ are finite with coprime orders.
The main issue that will be addressed in this paper is:
\begin{quote} \em
    Consider the collection of $A\cz{G}{A}$-invariant subgroups of $G$.
    How do these subgroups relate to one another and to the global
    structure of $G$?
\end{quote}

\noindent In the case that $G$ is solvable,
much is known.
A typical result is the following:
\begin{UnnumberedTheorem}[see {\cite[\S36]{A}} or \cite{F1}]
    Assume that $A$ has prime order $r$,
    that $G$ is solvable and that $H$ is an $A\cz{G}{A}$-invariant
    subgroup of $G$ with $H = [H,A]$.
    \begin{enumerate}
        \item[(a)]  Let $p$ be a prime.
                    If $p=2$ and $r$ is a Fermat prime assume that the
                    Sylow $2$-subgroups of $G$ are abelian.
                    Then \[
                        \op{p}{H} \leq \op{p}{G}.
                    \]

        \item[(b)]  If $H = \oupper{2}{H}$ then \[
                        \op{2}{H} \leq \op{2}{G}.
                    \]
    \end{enumerate}
\end{UnnumberedTheorem}
\noindent Thus, nearly always,
the Fitting subgroup of $H$ is contained in the Fitting subgroup of $G$.
This result is central to the author's proof of the Solvable Signalizer Functor Theorem.

In the theory of arbitrary finite groups,
attention is focussed on the generalized Fitting subgroup and components.
We shall introduce the notions of $A$-quasisimple group,
$A$-component and $(A, \mbox{sol})$-component.
The theory developed will revolve around these notions.
Basic properties of $A$-quasisimple groups will be established and the
main results will be stated and proved in \S\ref{l}.
This paper concludes with an application to the study of nonsolvable signalizer functors.
A precursor to this work is \cite{F3} where the author began the development
of the theory, but without a $K$-group hypothesis.

One issue that appears to be fundamental is the following:
let $R$ be a group of prime order $r$ that acts on the $r'$-group $G$
and let $V$ be a faithful completely reducible $RG$-module over a field.
Then $\cz{V}{R}$ is a module for $\cz{G}{R}$.
Let \[
    K = \kernelon{\cz{G}{R}}{\cz{V}{R}}.
\]
In \cite{F1} this situation is analyzed completely in the case that $G$ is solvable.
In a precisely defined sense,
it is shown that $K$ is almost subnormal in $G$.
We shall partially extend this result to arbitrary $G$.
In \S\ref{m} it will be shown that every component of $K$
is in fact a component of $G$.

The $K$-group hypothesis is somewhat of a departure from the previous
work of the author and deserves some comment.
Firstly, when the new proof of the Solvable Signalizer Functor Theorem
was discovered,
the challenge of extending that work to the nonsolvable case proved irresistible.
Secondly, and looking towards the future,
this work highlights issues that are fundamental to the theory and gives
direction to a more abstract study of automorphisms.
Hence continuing the work begun in \cite{F3, F4, F5} for example.

Finally, it must be emphasized that this work would not have been
possible without the prior work of McBride \cite{McB1, McB2}.
For example the material in \S\ref{aqs} on $A$-quasisimple groups is a
partial reworking of some of this results.
Moreover McBride's work provided clues to the general theory developed in \S\ref{l} and \S\ref{a}.

\section{Definitions}\label{d}
Let $G$ be a finite group.
The reader is assumed to be familiar with the notions of
the Fitting subgroup, the set of components, the layer
and the generalized Fitting subgroup of $G$
denoted by $\ff{G}, \compp{G}, \layerr{G}$ and $\gfitt{G}$ respectively.
See for example \cite{S}.
The notation $\sol{G}$ is used to denote the largest normal solvable
subgroup of $G$.
We define a number of variations on the notion of component.

\begin{Definition}\label{d:1}
    A \emph{sol-component} of $G$ is a perfect subnormal subgroup of $G$
    that maps onto a component of $G/\sol{G}$.
    The set of sol-components of $G$ is denoted by \[
        \compsol{G}
    \]
    and we define \[
        \layersol{G} = \listgen{ \compsol{G} }.
    \]
\end{Definition}
\noindent The sol-components of $G$ are characterized as being the
minimal nonsolvable subnormal subgroups of $G$.

The following lemma collects together the basic properties of sol-components.

\begin{Lemma}\label{d:2}
    Let $G$ be a finite group.
    \begin{enumerate}
        \item[(a)]  $\compp{G} \subseteq \compsol{G}$ and $\layerr{G} \normal \layersol{G}$.

        \item[(b)]  $K \in \compsol{G}$ if and only if
                    $K \subnormal G$, $K$ is perfect and $K/\sol{K}$ is simple.

        \item[(c)]  Let $K \in \compsol{G}$ and $S \subnormal G$.
                    Then
                    \begin{enumerate}
                        \item[(i)]  $K \leq S$; or
                        \item[(ii)] $[K,S] \leq K \cap S \leq \sol{K}$
                                    and $S \leq \n{G}{K}$.
                    \end{enumerate}

        \item[(d)]  $\sol{G}$ normalizes every sol-component of $G$.

        \item[(e)]  Suppose that $K$ and $L$ are distinct sol-components of $G$.
                    Then $K$ and $L$ normalize each other and
                    $[K,L] \leq \sol{K} \cap \sol{L} \subnormal \sol{G}$.

        \item[(f)]  Set $\br{G} = G/\sol{G}$.
                    The map $K \mapsto \br{K}$ defines a bijection
                    $\compsol{G} \longrightarrow \compp{\br{G}}$.
                    The inverse is given as follows:
                    if $\br{K} \in \compp{\br{G}}$,
                    let $L$ be the full inverse image of $\br{K}$ in $G$
                    and consider $L^{(\infty)}$.
    \end{enumerate}
\end{Lemma}
\noindent The proof is left as an exercise for the reader.
See for example Lemma~\ref{p:2}.

\begin{Definition}\label{d:3}\mbox{}
    \begin{itemize}
        \item   $G$ is \emph{constrained} if $\layerr{G} = 1$.
        \item   $G$ is \emph{semisimple} if $G = \layerr{G}$.
    \end{itemize}
\end{Definition}
\noindent Recall that $\gfitt{G} = \ff{G}\layerr{G}$ and that $\cz{G}{\gfitt{G}} = \zz{\ff{G}}$.
Thus $G$ is constrained if and only if $\gfitt{G} = \ff{G}$
if and only if $\cz{G}{\ff{G}} \leq \ff{G}$.
It is straightforward to show that any sol-component of $G$ is either
constrained or semisimple.

Next we bring into play a group $A$ that acts as a group of automorphisms on $G$.
It is convenient to use the language of groups with operators.
Thus $G$ is $A$-simple if $G$ is nonabelian and the only $A$-invariant normal
subgroups of $G$ are $1$ and $G$.
This implies that $G$ is a direct product of simple groups that are
permuted transitively by $A$.

Recall that $G$ is quasisimple if $G$ is perfect and $G/\zz{G}$ is simple.

\begin{Definition}\label{d:4}
    $G$ is \emph{$A$-quasisimple} if $G$ is perfect and $G/\zz{G}$ is $A$-simple.
\end{Definition}

\noindent It is straightforward to show that $G$ is $A$-quasisimple if and only if
$G$ is the central product of quasisimple groups that are permuted transitively by $A$.
Equivalently, $G = \layerr{G}$ and $A$ is transitive on $\compp{G}$.

Trivially, $A$ acts on the sets $\compp{G}$ and $\compsol{G}$.

\begin{Definition}\label{d:5}\mbox{}
    \begin{itemize}
        \item   An \emph{$A$-component} of $G$ is the subgroup generated by
                an orbit of $A$ on $\compp{G}$.

        \item   An \emph{$(A, \mbox{sol})$-component} of $G$ is the subgroup
                generated by an orbit of $A$ of $\compsol{G}$.
    \end{itemize}
    The sets of $A$-components and $(A, \mbox{sol})$-components of $G$
    are denoted by \[
        \comp{A}{G} \mbox{ and } \comp{A, \mbox{sol}}{G}
    \]
    respectively.
\end{Definition}

\noindent The $A$-components of $G$ are the $A$-quasisimple subnormal subgroups of $G$.
The $(A, \mbox{sol})$-components of $G$ are the minimal $A$-invariant
nonsolvable subnormal subgroups of $G$.
A result entirely analogous to Lemma~\ref{d:2} holds but for
$(A, \mbox{sol})$-components instead of sol-components.

\section{Preliminaries}\label{p}
\begin{Definition}\label{p:1}
    Suppose the group $G$ acts on the set $\Omega$.
    \begin{enumerate}
        \item[(a)]  The action is \emph{semiregular} if whenever
                    $\alpha \in \Omega, g \in G$ and $\alpha g = \alpha$
                    then $g = 1$.

        \item[(b)]  The action is \emph{regular} if it is semiregular and transitive.
    \end{enumerate}
\end{Definition}

\begin{Lemma}\label{p:2}
    Let $G$ be a group.
    \begin{enumerate}
        \item[(a)]  Let $K \in \compp{G}$ and $S \subnormal G$.
                    Then either $K \leq S$ or $[K,S] = 1$.

        \item[(b)]  Suppose $K$ is a perfect subnormal subgroup of $G$ and
                    that $S$ is a solvable subgroup of $G$ that is normalized by $K$.
                    Then $S \leq \n{G}{K}$.
                    If in addition $\sol{K} = \zz{K}$ then $[S,K] = 1$.
    \end{enumerate}
\end{Lemma}
\begin{proof}
    (a). This is \cite[6.5.2, p.142]{S}.

    (b). Without loss, $G = KS$.
    If $G = K$ the result is clear so assume $G \not= K$.
    Set $L = \listgen{K^{G}}$,
    so $L \not= G$ as $K \subnormal G$.
    Now $L = K(L \cap S)$ so by induction, $K \normal L$.
    Since $L \cap S$ is solvable and $K$ is perfect it follows that
    $K = L^{(\infty)} \characteristic L \normal G$,
    so $K \normal G$.

    Suppose also that $\sol{K} = \zz{K}$.
    Then $[K,S] \leq K \cap S \leq \sol{K} = \zz{K}$ whence $[K,S,K] = 1$.
    It follows from the Three Subgroups Lemma that $[S,K] = 1$.
\end{proof}

\begin{Definition}\label{p:3}
    The group $A$ \emph{acts coprimely} on the group $G$ if
    $A$ acts on $G$; the orders of $A$ and $G$ are coprime;
    and $A$ or $G$ is solvable.
\end{Definition}

\begin{Theorem}[Coprime Action]\label{p:4}
    Suppose the group $A$ acts coprimely on the group $G$.
    \begin{enumerate}
        \item[(a)]  $G = \cz{G}{A}[G,A]$ and $[G,A] = [G,A,A]$.

        \item[(b)]  If $G$ is abelian then $G = \cz{G}{A} \times [G,A]$.

        \item[(c)]  Suppose $N$ is an $A$-invariant normal subgroup of $G$.
                    Set $\br{G} = G/N$.
                    Then $\cz{\br{G}}{A} = \br{\cz{G}{A}}$.

        \item[(d)]  For each prime $p$ there exists an $A$-invariant
                    Sylow $p$-subgroup of $G$.
                    Every $A$-invariant $p$-subgroup is contained in an $A$-invariant
                    Sylow $p$-subgroup of $G$.
                    Moreover, $\cz{G}{A}$ acts transitively by conjugation on the
                    collection of $A$-invariant Sylow $p$-subgroups of $G$.

        \item[(e)]  Suppose $G = XY$ where $X$ and $Y$ are $A$-invariant subgroup of $G$.
                    Then $\cz{G}{A} = \cz{X}{A}\cz{Y}{A}$.

        \item[(f)]  If $[\gfitt{G},A] = 1$ then $[G,A] = 1$.

        \item[(g)]  Suppose that $N$ is an $A$-invariant normal Hall-subgroup of $G$
                    and that $N$ or $G/N$ is solvable.
                    Then $G$ possesses an $A$-invariant complement to $N$.
                    All such complements are conjugate under the action of $\cz{G}{A}$.
    \end{enumerate}
\end{Theorem}

\newcommand{\CoprimeActionComm}{Coprime Action(a)}
\newcommand{\CoprimeActionQuot}{Coprime Action(c)}
\newcommand{\CoprimeActionSyl}{Coprime Action(d)}
\newcommand{\CoprimeActionProd}{Coprime Action(e)}
\newcommand{\CoprimeActionGFitt}{Coprime Action(f)}

\begin{proof}
    For (a),\ldots,(e) see \cite[p.184--188]{S}.

    (f). We have $[G,A] \leq \cz{G}{\gfitt{G}} \leq \gfitt{G}$
    so $[G,A,A] = 1$.
    Apply (a).

    (g).\ This follows by applying the Schur-Zassenhaus Theorem and a Frattini argument
    to the semidirect product $AG$.
\end{proof}

\begin{Lemma}\label{p:5}
    Suppose the group $A$ acts on the perfect group $K$ and that
    $A$ acts trivially on $K/\zz{K}$.
    Then $A$ acts trivially on $K$.
\end{Lemma}
\begin{proof}
    We have $[K,A,K] \leq [\zz{K},K] = 1$ and similarly $[A,K,K] = 1$.
    The Three Subgroups Lemma forces $[K,K,A] = 1$.
    Since $K$ is perfect, the result follows.
\end{proof}

\begin{Lemma}\label{p:6}
    Let $A$ be a group that acts on the group $G$.
    Suppose that $G = K_{1} \times \cdots \times K_{n}$
    where $\listset{K_{1},\ldots,K_{n}}$ is a collection of subgroups
    that is permuted transitively by $A$.
    For each $i$ let $\pi_{i}:G \longrightarrow K_{i}$ be the projection map
    and set $B = \n{A}{K_{1}}$.
    Then \[
        \cz{G}{A} \isomorphic \cz{G}{A}\pi_{1} = \cz{K_{1}}{B}.
    \]
\end{Lemma}
\begin{proof}
    Let $c \in \cz{G}{A}$.
    Then there exist unique $c_{i} \in K_{i}$ such that $c = c_{1}\cdots c_{n}$,
    in fact $c_{i} = c\pi_{i}$.
    Now $A$ permutes the set $\set{K_{i}}{c_{i}\not=1}$ so as $A$ is transitive
    on $\listset{K_{1},\ldots,K_{n}}$,
    this set is either empty of equal to $\listset{K_{1},\ldots,K_{n}}$.
    It follows that the map $c \mapsto c_{1}$ is a monomorphism.
    This proves the isomorphism.
    Also, $A$ permutes $\listset{c_{1}, \ldots, c_{n}}$
    whence $c_{1} \in \cz{K_{1}}{B}$.
    Thus $\cz{G}{A}\pi_{1} \leq \cz{K_{1}}{B}$.

    Suppose now that we are given $c_{1} \in \cz{K_{1}}{B}$.
    For each $i$ choose $a_{i} \in A$ with $K_{i} = K_{1}^{a_{1}}$,
    so $\listset{a_{1}, \ldots, a_{n}}$ is a right transversal to $B$ in $A$.
    Define $c_{i} = c_{1}^{a_{i}} \in K_{i}$ and set $c = c_{1} \cdots c_{n}$.
    A simple argument shows that $A$ permutes $c_{1}, \ldots, c_{n}$,
    so as $[K_{i},K_{j}] = 1$ for all $i \not= j$ we have $c \in \cz{G}{A}$.
    Then $c\pi_{1} = c_{1}$ so $\cz{K_{1}}{B} \leq \cz{G}{A}\pi_{1}$.
    The proof is complete.
\end{proof}

We use the symbol $*$ to denote a central product.
Thus $G = H*K$ means $G = HK$ and $[H,K] = 1$.

\begin{Lemma}\label{p:7}
    Let $A$ be a group that acts coprimely on the group $K$.
    Suppose $K = K_{1} * \cdots * K_{n}$ for some $A$-invariant collection
    $\listset{K_{1}, \ldots, K_{n}}$ of subgroups of $K$ on which $A$
    acts regularly.
    Then $\cz{K}{A} \isomorphic K_{1}/Z$ for some subgroup
    $Z \leq \zz{K_{1}} \cap \zz{K_{2} * \cdots * K_{n}}$.
\end{Lemma}
\begin{proof}
    For each $i$ let $a_{i}$ be the unique member of $A$ with $K_{i} = K_{1}^{a_{i}}$,
    so $a_{1} = 1$.
    The map $\tau:k \mapsto k^{a_{1}}\cdots k^{a_{n}}$ is a
    homomorphism $K_{1} \longrightarrow \cz{K}{A}$.
    If $k \in \kernel{\tau}$ then
    $k = k^{a_{1}} = (k^{a_{2}}\cdots k^{a_{n}})^{-1} \in K_{1} \cap (K_{2} * \cdots * K_{n})
    \leq \zz{K_{1}} \cap \zz{K_{2} * \cdots * K_{n}}$.
    In order to complete the proof,
    it suffices to show that $\tau$ is surjective.

    Consider the external direct product $\widetilde{K} = K_{1} \times\cdots\times K_{n}$
    and the map $\sigma:\widetilde{K} \longrightarrow K$ defined by
    $(k_{1}, \ldots, k_{n})\sigma = k_{1}\cdots k_{n}$.
    Then $A$ acts coprimely on $\widetilde{K}$ and $\sigma$ is an $A$-epimorphism.
    By \CoprimeActionQuot, $\cz{\widetilde{K}}{A}\sigma = \cz{K}{A}$.
    Visibly $\cz{\widetilde{K}}{A} = \set{ (k^{a_{1}},\ldots,k^{a_{n}}) }{k \in K_{1}}$
    and the proof is complete.
\end{proof}

\begin{Lemma}\label{p:8}
    Let $A$ be a group that acts coprimely on the group $X$.
    Suppose that $AX$, the semidirect product of $X$ with $A$,
    acts on the set $\Omega$ and that $A$ acts transitively on $\Omega$.
    Then $X$ acts trivially on $\Omega$.
\end{Lemma}
\begin{proof}
    Choose $\alpha \in \Omega$.
    Let $p \in \primes{X}$.
    Now $AX = A\stab{AX}{\alpha}$ because $A$ is transitive.
    As $A$ is a $p'$-group it follows that $\stab{AX}{\alpha}$
    contains a Sylow $p$-subgroup $P$ of $AX$.
    Now $X$ is a normal Hall-subgroup of $AX$,
    whence $P \leq X$.
    It follows that $X \leq \stab{AX}{\alpha}$.
    Now $\alpha$ was arbitrary,
    so $X$ acts trivially on $\Omega$.
\end{proof}

\begin{Lemma}\label{p:9}
    Let $\field{F}$ be a field, $G$ a group and $V$ an $\pring{F}{G}$-module.
    \begin{enumerate}
        \item[(a)]  Suppose that $\fieldchar{F}$ does not divide $\card{G}$.
                    Then \[
                        V = \cz{V}{G} \oplus [V,G].
                    \]

        \item[(b)]  Suppose $V$ is faithful and $\fieldchar{F} = p$.
                    Then \[
                        \op{p}{G} = \bigcap\cz{G}{U}
                    \]
                    where $U$ ranges over the irreducible constituents  of $V$
                    and $\op{p}{G}$ is defined to be $1$ if $p=0$.
    \end{enumerate}
\end{Lemma}
\begin{proof}
    (a). By Maschke's Theorem, $V$ is a direct sum of irreducible submodules.
    Then $\cz{V}{G}$ is the sum of those submodules that are trivial and $[V,G]$
    is the sum of those modules that are nontrivial.

    (b). Suppose $p=0$.
    Then we may write $V$ as a direct sum of irreducible submodules,
    whence the intersection acts trivially on $V$.
    Suppose $p>0$.
    If $U$ is any irreducible $\pring{F}{G}$-module then $\cz{U}{\op{p}{G}} \not= 0$
    whence $\op{p}{G} \leq \cz{G}{U}$.
    Thus $\op{p}{G}$ is contained in the intersection.
    Let $q$ be a prime not equal to $p$ and let $Q$ be a Sylow $q$-subgroup of the intersection.
    By considering a composition series for $V$,
    we have $[V,Q, \ldots, Q] = 0$ and then (a),
    with $Q$ in the role of $G$,
    implies $[V,Q] = 0$.
    Then $Q = 1$ and we deduce that the intersection is a $p$-group.
\end{proof}

\begin{Lemma}\label{p:10}
    Let $R$ be a group of prime order $r$ that acts on the $q$-group $Q$
    with $q \not= r$ and $[Q,R] \not= 1$.
    Let $V$ be an $\pring{F}{RQ}$-module where $\field{F}$ is a field
    with $\fieldchar{F} \not= q$.
    Assume that $[Q,R]$ acts nontrivially on $V$.
    If $q=2$ and $r$ is a Fermat prime assume that $Q$ is abelian.
    Then $\pring{F}{R}$ is a direct summand of $V_{R}$.
    In particular $\cz{V}{R} \not= 0$.
\end{Lemma}
\begin{proof}
    By \CoprimeActionComm\ we may assume $Q = [Q,R]$.
    Apply \cite[Theorem~5.1]{F1}.
\end{proof}

The following is an easy special case of the main result of \cite{F1}.

\begin{Lemma}\label{p:11}
    Let $r,t$ and $p$ be primes.
    Suppose the group $R \times S$ acts on the group $T$ and that
    $V$ is an $\pring{F}{RST}$-module with $\field{F}$ a field of characteristic $p$.
    Assume that:
    \begin{enumerate}
        \item[(i)]  $\card{R} = r$, $S$ is an $r'$-group, $T$ is a $t$-group and $t\not= p$.
        \item[(ii)] $T = [T,S]$.
        \item[(iii)]$[\cz{V}{R},S] = 0$.
        \item[(iv)] If $T$ is nonabelian then $[\cz{V}{R},\cz{T}{R}] = 0$ and $t\not=2$.
    \end{enumerate}
    Then $[V,[T,R]] = 0$.
\end{Lemma}
\begin{proof}
    By \cite[Lemma~2.2]{F1} we may assume that $\field{F}$ is algebraically closed.
    Now $V = \cz{V}{[T,R]} \oplus [V,[T,R]]$ by Theorem~\ref{p:9}(a)
    and $[T,R] \normal RST$ so $[V,[T,R]]$ is an $RST$-module,
    hence we may suppose that $\cz{V}{[T,R]} = 0$ and moreover that $T$ acts
    faithfully on $V$.
    Let $V_{1}, \ldots, V_{n}$ be the homogeneous components for $\zz{T}$.
    Then $T$ normalizes each $V_{i}$ and $RS$ permutes the $V_{i}$ amongst themselves.
    Since $t\not=p$ we have $V = V_{1} \oplus\cdots\oplus V_{n}$.

    Suppose that $R$ does not normalize each $V_{i}$.
    Then without loss $\listset{V_{1},\ldots,V_{r}}$ is an orbit for the action of $R$.
    Set $W = V_{1} \oplus\cdots\oplus V_{r}$ so $\cz{W}{R}$ is a diagonal subspace of $W$.
    By assumption $[\cz{W}{R},S] = 0$ so $S$ permutes the $V_{i}$
    onto which $\cz{W}{R}$ projects nontrivially.
    We deduce that $S$ permutes $\listset{V_{1},\ldots,V_{r}}$.
    Lemma~\ref{p:8} implies that $S$ normalizes each $V_{i}$, $1 \leq i \leq r$.
    Then as $[\cz{W}{R},S] = 0$ it follows that $S$ centralizes $V_{1}$.
    But $T = [T,S]$ so $T$ centralizes $V_{1}$,
    contrary to $\cz{V}{[T,R]} = 0$.
    We deduce that $R$ normalizes each $V_{i}$.

    Choose $i$ with $1 \leq i \leq n$.
    Now $V_{i}$ is a homogeneous component for $\zz{T}$ and $\field{F}$
    is algebraically closed so $\zz{T}$ acts as scalar multiplication on $V_{i}$.
    Thus $[\zz{T},R]$ is trivial on $V_{i}$.
    As $V = V_{1} \oplus\cdots\oplus V_{n}$ we deduce that $[\zz{T},R] = 1$.
    In particular, the conclusion has been established in the case that
    $T$ is abelian,
    hence we assume that $T$ is nonabelian.

    By assumption $[\cz{V}{R},\cz{T}{R}] = 0$ so $\cz{V}{R} \leq \cz{V}{\zz{T}}$.
    Also $t\not=2$ so as $\cz{V}{[T,R]} = 0$, Lemma~\ref{p:10} implies $\cz{V_{i}}{R} \not= 0$.
    Consequently $\cz{V_{i}}{\zz{T}} \not= 0$.
    Now $V_{i}$ is a homogeneous component for $\zz{T}$ whence $\zz{T}$
    is trivial on $V_{i}$.
    Since $V = V_{1} \oplus\cdots\oplus V_{n}$ it follows that $\zz{T} = 1$.
    Then $T = 1$ and the result is established in this case also.
\end{proof}

\begin{Lemma}\label{p:12}
    Suppose the group $A$ acts on the constrained group $G$.
    Then \[
        \ff{G} = \bigcap\cz{G}{V}
    \]
    where $V$ ranges over the $A$-chief factors of $G$ below $\ff{G}$.
\end{Lemma}
\begin{UnnumberedRemark}
    The $A$-chief factors of $G$ below $\ff{G}$ are by definition
    the quotients $X/Y$ where $X$ and $Y$ are $A$-invariant normal
    subgroups of $G$ with $Y < X \leq \ff{G}$ and $X/Y$ being the only
    nontrivial $A$-invariant normal subgroup of $X/Y$.
    In particular, $X/Y$ is an elementary abelian $p$-group for some prime $p$
    and an irreducible $\gfpring{p}{AG}$-module.
\end{UnnumberedRemark}
\begin{proof}
    If $1 < N \normal F$ with $F$ nilpotent then $[N,F] < N$.
    It follows that $\ff{G}$ is contained in the right hand side.
    To prove the opposite inclusion,
    it suffices to show that if $D$ is an $A$-invariant normal subgroup
    of $G$ with $[\ff{G},D,\ldots,D] = 1$ then $D \leq \ff{G}$.

    Suppose that $D' < D$.
    Then by induction, $D' \leq \ff{G}$ whence $[D',D,\ldots,D] = 1$.
    Thus $D$ is nilpotent.
    As $D \normal G$ we have $D \leq \ff{G}$ as desired.
    Hence we may assume that $D' = D$.
    We have $[\ff{G},D] = [D,\ff{G}]$ so $[\ff{G},D,D] = [D,\ff{G},D] \normal G$
    so $[D,D,\ff{G}]  \leq [\ff{G},D,D]$ by the Three Subgroups Lemma.
    Now $[\ff{G},D] = [D,\ff{G}] = [D,D,\ff{G}] \leq [\ff{G},D,D]$.
    As $[\ff{G},D,\ldots,D] = 1$ this forces $[\ff{G},D] = 1$.
    Since $G$ is constrained we have $D \leq \ff{G}$ and the proof is complete.
\end{proof}

\section{Properties of K-groups}\label{k}
The following result collects together all
the specific properties of $K$-groups that we shall use.

\begin{Theorem}\label{k:1}
    Let $K$ be a simple $K$-group and suppose $r$ is a prime
    that does not divide $\card{K}$.
    \begin{enumerate}
        \item[(a)]  The Sylow $r$-subgroups of $\aut{K}$ are cyclic.
    \end{enumerate}

    \noindent Suppose $R \leq \aut{K}$ has order $r$.
    Set $C = \cz{K}{R}$.

    \begin{enumerate}
        \item[(b)]  $C$ possesses a unique minimal normal subgroup $N$.
                    Except for the cases listed in Tables~1 and 2,
                    $C = N$ and $C$ is simple.
                    Either $\gfitt{C}$ is simple or $C$ is solvable.
                    If $C$ is solvable then the possibilities for $C$
                    are listed in Table~1.

        \item[(c)]  $K$ possesses a unique maximal $RC$-invariant solvable subgroup $S$.
                    Suppose $S \not= 1$.
                    The possibilities for $K$ are listed in Table~1;
                    $C$ is solvable; $C \leq S$;
                    and $S$ is maximal subject to being an $RC$-invariant proper subgroup of $K$.

        \item[(d)]  $C$ is contained in a unique maximal $R$-invariant subgroup $M$.
                    If $M \not= C$ then $M$ is solvable and
                    $K \isomorphic \ltwo{2^{r}}$ or $\sz{2^{r}}$.

        \item[(e)]  Suppose that $X$ is an $R$-invariant $r'$-subgroup of $\aut{K}$
                    and that $[C,X] = 1$.
                    Then $X = 1$.

        \item[(f)]  Suppose that $\widetilde{K}$ is quasisimple with
                    $\widetilde{K}/\zz{\widetilde{K}} \isomorphic K$,
                    that $\widetilde{R} \leq \aut{\widetilde{K}}$ has order $r$
                    and that $V$ is a faithful $\pring{F}{\widetilde{R}\widetilde{K}}$-module
                    for some field $\field{F}$.
                    \begin{enumerate}
                        \item[(i)]  $\pring{F}{\widetilde{R}}$ is a direct summand of $V_{\widetilde{R}}$.
                                    In particular $\cz{V}{\widetilde{R}} \not= 0$.

                        \item[(ii)] Suppose $V$ is irreducible.
                                    Then $\layerr{\cz{\widetilde{K}}{\widetilde{R}}}$
                                    acts faithfully on $\cz{V}{\widetilde{R}}$.
                    \end{enumerate}
    \end{enumerate}
\end{Theorem}

\begin{table}
    \[\begin{array}{c|cccc}
        K   &   \ltwo{2^{r}}    &   \ltwo{3^{r}}  &   \sz{2^{r}}  &   \uthree{2^{r}} \\ \hline
        C   &   \ltwo{2} \isomorphic 3:2       &   \ltwo{3} \isomorphic 2^{2}:3 %
            &   \sz{2} \isomorphic 5:4     &   \uthree{2} \isomorphic 3^{2}:Q_{8}    \\
        N   &   3               &   2^{2}         &     5         &     3^{2} \\
        \card{C:N}  & 2         &   3             &     4         &      8    \\
        S   &   (2^{r}+1):2     &   C             & (2^{r}+2^{\frac{1}{2}(r-1)}\epsilon+1):4 & C \\
        \out{K} & r             &   2 \times r    &     r         &     3 \times r
    \end{array}\]
    where $\epsilon = 1$ if $r \equiv \pm 1 \bmod 8$ and $\epsilon = -1$ if $r \equiv \pm 3 \bmod 8$.\\
    $K:H$ indicates a Frobenius group with kernel $K$ and complement $H$.\
    \medskip
    \caption{}
\end{table}

\begin{table}
    \[\begin{array}{c|cccc}
        K   &   \spfour{2^{r}}      &   \twistedgtwo{3^{r}}     &   \gtwo{2^{r}}   & \twistedffour{2^{r}} \\ \hline
        C   &   \spfour{2}          &   \twistedgtwo{3}         &   \gtwo{2}       & \twistedffour{2}     \\
        N   &   \spfour{2}' \isomorphic \alt{6} \isomorphic \ltwo{9} %
                    & \twistedgtwo{3}' \isomorphic \ltwo{8} & \gtwo{2}' \isomorphic \uthree{3} & \twistedffour{2}' \\
        \card{C:N}  &   2   &   3   &   2   &   2
    \end{array}\]
    \medskip
    \caption{}
\end{table}
\begin{proof}[Proof of Theorem~\ref{k:1}(a),\ldots,(e)]
    (a). This is \cite[Theorem~7.1.2, p.336]{GLS3}.

    (b). This is \cite[Theorem~2.2.7, p.38]{GLS3}.

    (c). This is \cite[Theorem~7.1.9, p.340]{GLS3}.

    (d). This is the main result of \cite{BGL}.

    (e). This is established in the third paragraph of the proof
    of \cite[Theorem~7.1.4, p.337]{GLS3}.
\end{proof}

\noindent The author is indebted to Richard Lyons for the proof of the
following lemma.

\begin{Lemma}\label{k:2}
    Let $R$ be a group of prime order $r$ that acts nontrivially
    and coprimely on the simple $K$-group $K$.
    Then there exists a prime power $q$ and $R$-invariant subgroups
    $L_{1},\ldots,L_{n}$ such that \[
        K = \listgen{L_{1},\ldots,L_{n}}
    \]
    and for each $i$,
    the action of $R$ on $L_{i}$ is nontrivial
    and $L_{i} \isomorphic \ltwo{q^{r}}, \sltwo{q^{mr}}, m = 1,2,3$ or $\sz{q^{r}}$.
\end{Lemma}
\begin{proof}
    Since $R$ acts nontrivially and coprimely on $K$ it follows that $K \in \chev{p}$
    for some prime $p$ and that $R$ is generated by a field automorphism,
    by \cite[7.1.2]{GLS3}.
    Then $K = \chevgroup{d}{q^{r}}$ where $q = p^{k}$ for some $k$.
    Since the Sylow $r$-subgroups of $\aut{K}$ are cyclic,
    the image of $R$ in $\aut{K}$ is determined up to conjugacy.
    Then replacing $R$ by a conjugate if necessary,
    we may assume that $R$ has a generator $\rho$ which is a field automorphism
    in the sense of \cite[Sec.\ 10]{St} (cf. \cite[2.5.1]{GLS3}).
    That is $\rho$ transforms a set of Chevalley generators
    $x_{\alpha}(t)$ or $x_{\alpha}(t,u)$, etc.\
    by taking them to $x_{\alpha}(t^{\psi}), x_{\alpha}(t^{\psi},u^{\psi})$, etc.,
    where $\psi$ is an automorphism of $\overline{\gf{q}}$.
    Thus for each root $\alpha$,
    $R$ normalizes the (twisted) rank one group $\listgen{X_{\alpha}, X_{-\alpha}}$.
    Such rank one groups generate $K$ so we may assume that $K$ has rank one.
    If $K \isomorphic A_{1}(q^{mr})$ or $\sz{q^{r}}$
    there is nothing to prove.
    If $K \isomorphic \twistedgtwo{q^{r}}$ then $R$ centralizes some $S \in \syl{2}{K}$,
    so $R$ normalizes each $\cz{K}{t} \isomorphic \listgen{t} \times \ltwo{q^{r}}, t \in S\nonid$,
    and $K = \gen{\layerr{\cz{K}{t}}}{t \in S\nonid}$ since otherwise the right hand side
    would be strongly embedded in $K$.
    If $K \isomorphic \uthree{q^{r}}$,
    then we may take the sesquilinear form to have matrix the $3 \times 3$ identity matrix,
    and $\rho$ to be the automorphism $t \mapsto t^{q}$ on all matrix entries.
    Then $K = \listgen{K_{12}, K_{23}}$ where $K_{12}$ and $K_{23}$ are block-diagonal
    copies of $\operatorname{SU}_{2}(q^{r})$.
    As $K_{12}$ and $K_{23}$ are $\rho$-invariant,
    the proof is complete.
\end{proof}

\begin{proof}[Proof of Theorem~\ref{k:1}(f)(i)]
    Let $p = \fieldchar{F}$.
    By Lemma~\ref{p:10},
    it suffices to show that $\widetilde{K}$ possesses an $\widetilde{R}$-invariant
    abelian $p'$-subgroup on which $\widetilde{R}$ acts nontrivially.
    The inverse image in $\widetilde{K}$ of any cyclic subgroup of $K$ is abelian.
    Hence it suffices to show that $K$ possesses an $R$-invariant cyclic
    $p'$-subgroup on which $R$ acts nontrivially.

    By Lemma~\ref{k:2} we may suppose that $K = \ltwo{q^{r}}$ or $\sz{q^{r}}$
    for some prime power $q$.
    Suppose $K = \ltwo{q^{r}}$.
    Set $d = (2,q-1)$.
    Then $K$ possesses $R$-invariant cyclic subgroups of orders
    $(q^{r}-1)/d$ and $(q^{r}+1)/d$ on which $R$ acts nontrivially.
    These orders are coprime,
    so one will be coprime to $p$.
    Suppose $K = \sz{q^{r}}$.
    Then $q = 2^{n}$ for some $n$.
    By \cite{Suz},
    $K$ possesses $R$-invariant cyclic subgroups of orders
    $2^{nr}+2^{(nr+1)/2}+1$ and $2^{nr}-2^{(nr+1)/2}+1$
    on which $R$ acts nontrivially.
    Again, one of these numbers is coprime to $p$.
\end{proof}

\begin{Lemma}\label{k:3}
    Let $R$ be a group of prime order $r$ that acts nontrivially
    and coprimely on the simple $K$-group $K$.
    Let $p$ be a prime.
    Then there exists a prime $t \not\in \listset{2,p}$ and an
    $R$-invariant dihedral group $D \leq K$ of order $2t$
    such that $R$ is nontrivial on $\op{t}{D}$ and $\cz{K}{R}$
    contains an involution of $D$.
    If $K \not\isomorphic \ltwo{2^{r}}$ and $\sz{2^{r}}$ then $D$
    may be chosen such that $\cz{K}{R}'$ contains an involution of $D$.
\end{Lemma}
\begin{proof}
    We begin by considering the special cases
    $K \isomorphic \ltwo{q^{r}}$ or $\sz{q^{r}}$ for some prime power $q$.
    Suppose that $K \isomorphic \ltwo{q^{r}}$.
    Choose $\epsilon \in \listset{-1,1}$,
    set $\delta = 1$ if $q$ is even and $\delta = 1/2$ if $q$ is odd.
    Now $\cz{K}{R} \isomorphic \ltwo{q}$ and $K$ possesses an $R$-invariant
    cyclic subgroup $X$ with order $\delta(q^{r}-\epsilon)$ that is
    inverted by an involution $z \in \cz{K}{R}$ and satisfies
    $\card{\cz{X}{R}} = \delta(q-\epsilon)$.
    Now \[
        \delta(q^{r}-\epsilon) = %
        \delta(q-\epsilon)\left( (\epsilon q)^{r-1} + \cdots + 1 \right)
    \]
    and $X$ possesses a subgroup $Y$ of order $(\epsilon q)^{r-1} + \cdots + 1$.
    Then $Y$ is $R$-invariant, inverted by $z$,
    has odd order and $\cz{Y}{R} = 1$.
    The two choices for $Y$,
    depending on the choice of $\epsilon$,
    have coprime orders.
    Hence we may choose $\epsilon$ such that $p \not\in \primes{Y}$.
    Choose a prime $t \in \primes{Y}$ and let $T$ be the subgroup of $Y$
    with order $t$.
    Set $D = T\listgen{z}$.
    Recall that $\cz{K}{R} \isomorphic \ltwo{q}$.
    If $q>3$ then $\cz{K}{R}$ is simple,
    whence $z \in \cz{K}{R}'$.
    If $q=3$ then $\ltwo{q} \isomorphic 2^{2}:3$
    and again $z \in \cz{K}{R}'$.

    Suppose that $K \isomorphic \sz{q^{r}}$.
    Then $q = 2^{n}$ for some odd $n$.
    Again choose $\epsilon \in \listset{-1,1}$.
    Now $\cz{K}{R} \isomorphic \sz{q}$ and by \cite{Suz},
    $K$ contains an $R$-invariant cyclic Hall-subgroup $X$ of order
    $2^{nr}+\epsilon 2^{(nr+1)/2}+1$ that is inverted
    by an involution $z \in \cz{K}{R}$.
    Note that $X$ has odd order and is not centralized by $R$.
    Set $Y = [X,R] \not= 1$.
    Then $Y$ is inverted by $z$.
    As previously, we may choose $\epsilon$ such that $Y$ is a $p'$-group.
    Choose $t \in \primes{Y}$ and let $T$ be the subgroup of $Y$
    with order $t$.
    Set $D = T\listgen{z}$.
    If $q>2$ then $\cz{K}{R}$ is simple so $z \in \cz{K}{R}'$.

    We now consider the general case.
    Using Lemma~\ref{k:2} and what we have just done,
    there exists an $R$-invariant dihedral subgroup $D \leq K$
    with order $2t$ for some $t \not\in \listset{2,p}$,
    $R$ is nontrivial on $\op{t}{D}$ and $\cz{K}{R}$ contains an involution of $D$.
    It remains to prove the final assertion.
    If $\cz{K}{R}$ is simple then there is nothing further to prove.
    Hence we may assume that $K$ is one of the eight groups listed
    in Tables~1 and 2 of Theorem~\ref{k:1}.
    The cases $\ltwo{2^{r}}$ and $\sz{2^{r}}$ are excluded by hypothesis.
    The case $\ltwo{3^{r}}$ has been dealt with.
    If $K \isomorphic \uthree{2^{r}}$ then $\cz{K}{R} \isomorphic 3^{2}:Q_{8}$
    so $\cz{K}{R}'$ contains every involution of $\cz{K}{R}$.
    If $K \isomorphic \twistedgtwo{3^{r}}$ then $\cz{K}{R}'$ has odd index
    in $\cz{K}{R}$ so again $\cz{K}{R}'$ contains every involution of $\cz{K}{R}$.
    The remaining three cases require a little more work.

    Suppose $K \isomorphic \spfour{2^{r}}$ or $\gtwo{2^{r}}$.
    Then $K$ contains  an $R$-invariant subgroup
    $H \isomorphic \ltwo{2^{r}} \times \ltwo{2^{r}}$ with $R$
    acting nontrivially on each component.
    This is clear in the case $K \isomorphic \spfour{2^{r}}$
    and follows from \cite{C} in the case $K \isomorphic \gtwo{2^{r}}$.
    By what we have done previously,
    $H$ contains an $R$-invariant subgroup $D = D_{1} \times D_{2}$
    with each $D_{i}$ dihedral of order $2t$ for some prime $t \not\in \listset{2,p}$,
    each $D_{i}$ is $R$-invariant and $R$ acts nontrivially on $\op{t}{D_{i}}$.
    From Table~2 in Theorem~\ref{k:1} we have $\card{\cz{K}{R}:\cz{K}{R}'}=2$
    so $\cz{K}{R}'$ contains an involution $u \in D$.
    Choose $i$ such that $u$ inverts $\op{t}{D_{i}}$.
    Then $\op{t}{D_{i}}\listgen{u}$ is the desired dihedral subgroup.

    Suppose $K \isomorphic \twistedffour{2^{r}}$.
    By \cite{M}, $K$ contains an $R$-invariant subgroup
    $H \isomorphic \spfour{2^{r}}$ on which $R$ acts nontrivially.
    Apply the previously considered case.
\end{proof}

\begin{proof}[Proof of Theorem~\ref{k:1}(f)(ii)]
    Let $\widetilde{E} = \layerr{\cz{\widetilde{K}}{\widetilde{R}}}$,
    $\widetilde{X} = \kernelon{\widetilde{E}}{\cz{V}{\widetilde{R}}}$
    and let $E$ be the image of $\widetilde{E}$ in $\cz{K}{R}$.
    Since $\widetilde{K}/\zz{\widetilde{K}} = K$ we have $E = \layerr{\cz{K}{R}}$.

    Assume the result is false.
    Then $\widetilde{X} \not= 1$ whence $\widetilde{E} \not= 1, E \not= 1$ and
    Theorem~\ref{k:1} implies $K \not\isomorphic \ltwo{2^{r}}$ and $\sz{2^{r}}$.
    Also, $E$ is simple whence $\widetilde{E}$ is quasisimple
    and $\zz{\widetilde{E}} \leq \zz{\widetilde{K}}$.
    Since $\widetilde{X} \normal \widetilde{E}$ we have $\widetilde{X} \leq \zz{\widetilde{E}}$
    or $\widetilde{X} = \widetilde{E}$.
    Suppose that $\widetilde{X} \leq \zz{\widetilde{E}}$.
    By (f)(i) we have $0 \not= \cz{V}{\widetilde{R}} \leq \cz{V}{\widetilde{X}}$.
    Also $\cz{V}{\widetilde{X}}$ is a submodule because
    $\widetilde{X} \leq \zz{\widetilde{E}} \leq \zz{\widetilde{K}}$.
    This contradicts the irreducibility of $V$.
    We deduce that $\widetilde{X} = \widetilde{E}$.
    In particular, as $E = \cz{K}{R}'$ it follows that $\widetilde{X}$ maps onto $\cz{K}{R}'$.

    By Lemma~\ref{k:3} there exists a prime $t \not\in \listset{2, \fieldchar{F}}$
    and an $R$-invariant dihedral subgroup $D \leq K$ of order $2t$ such that
    $R$ is nontrivial on $\op{t}{D}$ and $\cz{K}{R}'$ contains an involution of $D$.
    Let $T = \op{t}{D}$ and choose $S \leq \cz{K}{R}' \cap D$ with order $2$.

    Let $\widetilde{S} \leq \widetilde{X}$ be a $2$-subgroup that maps onto $S$.
    Since $T$ is cyclic,
    the inverse image of $T$ in $\widetilde{K}$ is abelian.
    Let $\widetilde{T}$ be a Sylow $t$-subgroup of this inverse image.
    Then $\widetilde{T}$ is $\widetilde{R} \times \widetilde{S}$-invariant
    and $\widetilde{T}$ maps onto $T$.
    Let $\widetilde{T}_{0} = [\widetilde{T},\widetilde{S}]$.
    \CoprimeActionComm\ implies $\widetilde{T}_{0} = [\widetilde{T}_{0},\widetilde{S}]$.
    Note that $\widetilde{T}_{0}$ is $\widetilde{R}$-invariant since
    $[\widetilde{R},\widetilde{S}] = 1$.
    Now $T = [T,S]$ and $\widetilde{T}_{0}$ maps onto $T$ whence
    $[\widetilde{T}_{0},\widetilde{R}] \not= 1$ because $[T,R] \not= 1$.
    But $[\cz{V}{\widetilde{R}},\widetilde{S}] = 0$ so Lemma~\ref{p:11}
    implies $[\widetilde{T}_{0},\widetilde{R}] = 1$,
    a contradiction.
    The proof is complete.
\end{proof}

We close this section with some useful consequences of Theorem~\ref{k:4}.

\begin{Theorem}\label{k:4}
    Let $r$ be a prime and suppose the elementary abelian $r$-group $A$
    acts coprimely on the $K$-group $G$.
    \begin{enumerate}
        \item[(a)]  If $\cz{G}{A}$ is nilpotent or has odd order then $G$ is solvable.

        \item[(b)]  If $\cz{G}{A}$ is solvable then the noncyclic composition factors
                    of $G$ belong to $\badfour$.

        \item[(c)]  Let $K \in \comp{A}{G}$.
                    Then $\cz{G}{\cz{K}{A}} = \cz{G}{K}$.

        \item[(d)]  $\zz{\cz{G}{A}} \leq \sol{G}$.
    \end{enumerate}
\end{Theorem}
\begin{proof}
    (a),(b). Using \CoprimeActionQuot\ it follows that a minimal counterexample is $A$-simple.
    Thus $G = K_{1} \times\cdots\times K_{n}$ where $K_{1}, \ldots, K_{n}$ are simple
    subgroups that are permuted transitively by $A$.
    Let $B = \n{A}{K_{1}}$.
    Lemma~\ref{p:6} implies that \[
        \cz{G}{A} \isomorphic \cz{K_{1}}{B}.
    \]
    Apply Theorem~\ref{k:1}.

    (c). Trivially $\cz{G}{K} \leq \cz{G}{\cz{K}{A}}$.
    Set $Z = \cz{G}{\cz{K}{A}}$.
    Using \CoprimeActionQuot\ and Lemma~\ref{p:5} we may suppose that $\zz{\layerr{G}} = 1$.
    Then $\layerr{G}$ is the direct product of the $A$-components of $G$
    and $\cz{G}{A}$ permutes these $A$-components by conjugation.
    By (a), $\cz{K}{A} \not= 1$ so as $[Z,\cz{K}{A}] = 1$ it follows that $Z$ normalizes $K$.

    We have $K = K_{1} \times\cdots\times K_{n}$ where $K_{1},\ldots,K_{n}$
    are simple subgroups that are permuted transitively by $A$.
    Lemma~\ref{p:8} implies that $Z$ normalizes each $K_{i}$.
    For each $i$ let $\pi_{i}:K\longrightarrow K_{i}$ be the projection map
    and set $A_{i} = \n{A}{K_{i}}$.
    Let $c \in \cz{K}{A}$.
    Then $c = (c\pi_{1})\cdots(c\pi_{n})$.
    Since $[c,Z] = 1$ and $Z$ normalizes each $K_{i}$ it follows that $[c\pi_{i},Z] = 1$.
    Lemma~\ref{p:6} implies $\cz{K}{A}\pi_{i} = \cz{K_{i}}{A_{i}}$ so
    $[\cz{K_{i}}{A_{i}},Z] = 1$ and then Theorem~\ref{k:1}(a),(e) imply $[K_{i},Z] = 1$.
    Then $[K,Z] = 1$.

    (d). Set $\br{G} = G/\sol{G}$.
    Then $\cz{\br{G}}{\layerr{\br{G}}} = 1$.
    \CoprimeActionQuot\ and (c) imply $[\layerr{\br{G}},\br{\zz{\cz{G}{A}}}] = 1$
    whence $\zz{\cz{G}{A}} \leq \sol{G}$.
\end{proof}
\section{Direct Products}\label{dp}
We establish some notation relating to direct products
and present a lemma of McBride \cite[Lemma~5.10]{McB2}.
Throughout this section we assume:

\begin{Hypothesis}\label{dp:1} \mbox{}
    \begin{itemize}
        \item   $G = K_{1} \times\cdots\times K_{n}$ with each $K_{i}$
                a nonabelian simple group.

        \item   For each $i$, $\pi_{i}$ is the projection $G \longrightarrow K_{i}$.
    \end{itemize}
\end{Hypothesis}
\noindent We remark that the subgroups $K_{i}$ are the components of $G$
and are uniquely determined,
as are the projection maps.

\begin{Definition}\label{dp:2}
    Let $H$ be a subgroup of $G$.
    \begin{itemize}
        \item   $H$ is \emph{diagonal} if for each $i$ the projection map
                $H \longrightarrow K_{i}$ is an isomorphism.

        \item   $H$ is \emph{overdiagonal} if for each $i$ the projection map
                $H \longrightarrow K_{i}$ is an epimorphism.

        \item   $H$ is \emph{underdiagonal} if for each $i$ the projection map
                $H \longrightarrow K_{i}$ is not an epimorphism.
    \end{itemize}
\end{Definition}

\begin{Lemma}\label{dp:3}
    Suppose $H$ is an overdiagonal subgroup of $G$.
    Then there exists a unique partition
    $\listset{ \mathcal L_{1},\ldots,\mathcal L_{m} }$ of $\listset{K_{1},\ldots,K_{n}}$
    such that \[
        H = ( H \cap \listgen{\mathcal L_{1}} ) \times\cdots\times ( H \cap \listgen{\mathcal L_{m}} )
    \]
    and $H \cap \listgen{\mathcal L_{i}}$ is a
    diagonal subgroup of $\listgen{\mathcal L_{i}}$ for each $i$.
\end{Lemma}
\begin{proof}
    Choose $\mathcal L_{1} \subseteq \listset{K_{1},\ldots,K_{n}}$ minimal subject
    to $H \cap \listgen{\mathcal L_{1}} \not= 1$.
    Set $H_{1} = H \cap \listgen{\mathcal L_{1}} \normal H$.
    Choose $K_{i} \in \mathcal L_{1}$.
    The minimal choice of $\mathcal L_{1}$ implies $H_{1} \cap \kernel{\pi_{i}} = 1$.
    Thus $1 \not= H_{1}\pi_{i} \normal H\pi_{i} = K_{i}$ so the simplicity of $K_{i}$
    forces $H_{1}\pi_{i} = H\pi_{i} = K_{i}$.
    Then $H_{1}$ is diagonal in $\listgen{\mathcal L_{1}}$.
    Also $H = H_{1}(H \cap \kernel{\pi_{i}})$ so as
    $H_{1} \cap \kernel{\pi_{i}} = 1$ we obtain
    \[
        H = H_{1} \times (H \cap \kernel{\pi_{i}}).
    \]
    As $H_{1} \isomorphic K_{i}$ we see that $H_{1}$ is simple
    and then that $H \cap \kernel{\pi_{i}} = \cz{H}{H_{1}}$.
    Set \[
        G^{*} = \prod_{K_{j}\not\in\mathcal L_{1}} K_{j} %
                = \bigcap_{K_{i}\in\mathcal L_{1}} \kernel{\pi_{i}}.
    \]
    Then $H = H_{1} \times (H \cap G^{*})$.
    Now $H_{1}$ projects trivially into the direct factors of $G^{*}$
    so as $H$ is overdiagonal in $G$ it follows that $H \cap G^{*}$
    is overdiagonal in $G^{*}$.
    Induction yields $\mathcal L_{2},\ldots,\mathcal L_{m}$.

    Now $H = (H \cap \listgen{\mathcal L_{1}}) \times\cdots\times %
    (H \cap \listgen{\mathcal L_{m}})$ so each $H \cap \listgen{\mathcal L_{i}}$
    is a component of $H$.
    The components of a group are uniquely determined so the
    uniqueness of $\listset{\mathcal L_{1},\ldots,\mathcal L_{m}}$ follows.
\end{proof}

\begin{Lemma}\label{dp:4}
    Let $H$ be an overdiagonal subgroup of $G$.
    Then $\n{G}{H} = H$.
\end{Lemma}
\begin{proof}
    By the previous lemma we may assume that $H$ is diagonal.
    Since $H\pi_{1} = K_{1}$ we have \[
        \n{G}{H} = HN
    \]
    where $N = \n{G}{H} \cap \kernel{\pi_{1}}$.
    Now $[H,N] \leq H \cap \kernel{\pi_{1}} = 1$.
    For any $i$ we have $1 = [H\pi_{i},N\pi_{i}] = [K_{i},N\pi_{i}]$
    and so $N\pi_{i} \leq \zz{K_{i}} = 1$.
    This forces $N = 1$ and completes the proof.
\end{proof}

\section{$A$-quasisimple groups}\label{aqs}
Throughout this section we assume:

\begin{Hypothesis}\label{aqs:1} \mbox{}
    \begin{itemize}
        \item   $r$ is a prime and $A$ is an elementary abelian $r$-group.
        \item   $A$ acts coprimely on the $K$-group $K$.
        \item   $K$ is $A$-quasisimple.
    \end{itemize}
\end{Hypothesis}

\noindent We will establish a number of basic results on the
subgroup structure of $K$.
A central theme is the study of the $A\cz{K}{A}$-invariant subgroups of $K$.
Of course the subgroups $\cz{K}{B}$ for $B \leq A$ are examples.
It will develop that these comprise an almost complete list.
The results of this section may also be viewed as an extension of
Theorem~\ref{k:1} from simple groups to $A$-quasisimple groups.
A similar theory is also developed by McBride \cite{McB1,McB2}
but cast in a different language.

Let $K_{1},\ldots,K_{n}$ be the components of $K$.
Then \[
    K = K_{1} *\cdots * K_{n}
\]
and $A$ acts transitively on $\listset{K_{1},\ldots,K_{n}}$.
In particular, $K$ has a unique nonsolvable composition factor.

\begin{Definition}\label{aqs:2}
    The \emph{type} of $K$ is the isomorphism type of the
    unique nonsolvable composition factor of $K$.
\end{Definition}

\noindent Let $\br{K} = K/\zz{K}$,
so $\br{K}$ is $A$-simple and \[
    \br{K} = \br{K_{1}} \times\cdots\times \br{K_{n}}
\]
with each $\br{K_{i}}$ being simple.

\begin{Definition}\label{aqs:3}
    Let $H$ be an $A$-invariant subgroup of $K$.
    Then $H$ is \emph{underdiagonal, diagonal} or \emph{overdiagonal}
    in $K$ depending on whether $\br{H}$ has the respective property in $\br{K}$.
\end{Definition}

\noindent Note that since $H$ is $A$-invariant and $A$ is transitive
on $\listset{K_{1},\ldots,K_{n}}$ it follows that $H$ is either
underdiagonal or overdiagonal.

We fix the notation \[
    A_{\infty} = \kernel{(A \longrightarrow \sym{K_{1},\ldots,K_{n}})}.
\]
Since $A$ is abelian and transitive on $\listset{K_{1},\ldots,K_{n}}$
it follows that $A_{\infty} = \n{A}{K_{i}}$ for each $i$
and that the action of $A/A_{\infty}$ on $\listset{K_{1},\ldots,K_{n}}$ is regular.

\begin{Lemma}\label{aqs:4}
    $A_{\infty}/\cz{A}{K}$ acts faithfully on each $K_{i}$ and
    $\card{A_{\infty}/\cz{A}{K}} = 1$ or $r$.
\end{Lemma}
\begin{proof}
    Since $A$ is abelian and transitive on $\listset{K_{1},\ldots,K_{n}}$
    it follows that $\cz{A_{\infty}}{K_{i}} = \cz{A}{K_{1}*\cdots *K_{n}} = \cz{A}{K}$.
    Theorem~\ref{k:1} and Lemma~\ref{p:5} imply that the Sylow $r$-subgroups
    of $\aut{K_{i}}$ are cyclic and the result follows.
\end{proof}

Next we describe the structure of the subgroups $\cz{K}{B}$ for $B \leq A$.

\begin{Lemma}\label{aqs:5}
    Let $B \leq A$.
    \begin{enumerate}
        \item[(a)]  Suppose $B \cap A_{\infty} \leq \cz{A}{K}$.
                    Then there exists $Z \leq \zz{K_{1}} \cap \zz{K_{2}\cdots K_{n}}$
                    such that $\cz{K}{B}$ is isomorphic to the central product of
                    $\card{A:BA_{\infty}}$ copies of $K_{1}/Z$
                    that are permuted transitively by $A$.
                    In particular, $\cz{K}{B}$ is overdiagonal and $A$-quasisimple
                    with the same type as $K$.
                    If $K$ is $A$-simple then so is $\cz{K}{B}$.

        \item[(b)]  Suppose $B \cap A_{\infty} \not\leq \cz{A}{K}$.
                    Then there exists $Z \leq \zz{K_{1}} \cap \zz{K_{2}\cdots K_{n}}$
                    such that $\cz{K}{B}$ is isomorphic to the central product of
                    $\card{A:BA_{\infty}}$ copies of $\cz{K_{1}}{A_{\infty}}/Z$
                    that are permuted transitively by $A$.
                    In particular, $\cz{K}{B}$ is underdiagonal.
                    Either $\cz{K}{B}$ is solvable or $\gfitt{\cz{K}{B}}$ is $A$-quasisimple.
                    If $K$ is $A$-simple then either $\cz{K}{B}$ is solvable
                    or $\gfitt{\cz{K}{B}}$ is $A$-simple.

        \item[(c)]  If $B^{*} \leq A$ and $\cz{K}{B^{*}} = \cz{K}{B}$
                    then $B^{*}\cz{A}{K} = B\cz{A}{K}$.

        \item[(d)]  $\cz{A}{\cz{K}{B}} = B\cz{A}{K}$.
    \end{enumerate}
\end{Lemma}
\begin{proof}
    We may assume that $\cz{A}{K} = 1$.
    Then $\card{A_{\infty}} = 1$ or $r$ by Lemma~\ref{aqs:4}.

    (a). Let $m = \card{A:BA_{\infty}}$.
    Then $B$ has $m$ orbits on $\listset{K_{1},\ldots,K_{n}}$
    and these orbits are permuted transitively by $A$.
    Let $L_{1},\ldots,L_{m}$ be the subgroups of $K$ that
    are generated by these orbits.
    Then $K = L_{1}*\cdots *L_{m}$.
    \CoprimeActionProd\ implies that $\cz{K}{B} = \cz{L_{1}}{B}*\cdots*\cz{L_{m}}{B}$.
    The subgroups $\cz{L_{1}}{B},\ldots,\cz{L_{m}}{B}$
    are permuted transitively by $A$.
    Without loss, $L_{1} = K_{1} *\cdots * K_{l}$.
    Now $B \cap A_{\infty} = 1$ so $B$ is regular on $\listset{K_{1},\ldots,K_{l}}$.
    Apply Lemma~\ref{p:7}.

    (b). We have $\card{A_{\infty}} = r$ and so there exists $B_{0}$
    with $B = A_{\infty} \times B_{0}$.
    Now \[
        \cz{K}{A_{\infty}} = \cz{K_{1}}{A_{\infty}} *\cdots * \cz{K_{n}}{A_{\infty}}
    \]
    so applying an argument similar to that used in (a),
    with $B_{0}$ in place of $B$ and the $\cz{K_{i}}{A_{\infty}}$
    in place of the $K_{i}$,
    the first assertion follows.
    Since $\cz{K}{B} \leq \cz{K}{A_{\infty}}$,
    trivially $\cz{K}{B}$ is underdiagonal.
    The remaining assertions follow from Theorem~\ref{k:1}.

    (c). We have $\cz{K}{BB^{*}} = \cz{K}{B}$.
    Since a subgroup cannot be both underdiagonal and overdiagonal
    we have $BB^{*} \cap A_{\infty} = B \cap A_{\infty}$.
    Now (a) and (b) imply $\card{A:BB^{*}A_{\infty}} = \card{A:BA_{\infty}}$,
    whence $\card{BB^{*}} = \card{B}$ and $B^{*} \leq B$.
    Similarly, $B \leq B^{*}$ so $B = B^{*}$.

    (d). Apply (c) with $B^{*} = \cz{A}{\cz{K}{B}}$.
\end{proof}

The next result shows that modulo $\zz{K}$,
the subgroups just considered are the only
$A\cz{K}{A}$-invariant overdiagonal subgroups of $K$.

\begin{Lemma}\label{aqs:6}
    Suppose that $H$ is an $A\cz{K}{A}$-invariant overdiagonal subgroup of $K$.
    Then there exists $B \leq A$ such that $B \cap A_{\infty} \leq \cz{A}{K}$
    and \[
        H = \cz{K}{B}(H \cap \zz{K}).
    \]
    In particular, if $K$ is $A$-simple then $H = \cz{K}{B}$ and
    $H$ is $A$-simple with the same type as $K$.
\end{Lemma}
\begin{proof}
    Suppose the lemma has been established in the case that $K$ is $A$-simple.
    Set $\br{K} = K/\zz{K}$.
    \CoprimeActionQuot\ implies that $\cz{\br{K}}{A} = \br{\cz{K}{A}}$
    so $\br{H}$ is $\cz{\br{K}}{A}$-invariant,
    whence $\br{H} = \cz{\br{K}}{B}$ for some
    $B \leq A$ with $B \cap A_{\infty} \leq \cz{A}{\br{K}}$.
    Lemma~\ref{p:5} implies that $B \cap A_{\infty} \leq \cz{A}{K}$.
    Another application of \CoprimeActionQuot\ yields \[
        H\zz{K} = \cz{K}{B}\zz{K}.
    \]
    Lemma~\ref{aqs:5}(a) implies that $\cz{K}{B}$ is $A$-quasisimple.
    In particular it is perfect.
    Then $H' = (H\zz{K})' = (\cz{K}{B}\zz{K})' = \cz{K}{B}$ so
    $\cz{K}{B} \leq H \leq \cz{K}{B}\zz{K}$
    and then $H = \cz{K}{B}(H \cap \zz{K})$.
    Hence we may suppose that $K$ is $A$-simple.

    Consider the case that $H$ is diagonal.
    Lemma~\ref{dp:4} implies $\cz{K}{A} \leq H$.
    Now $H \isomorphic K_{1}$ so $H$ is simple.
    Set $B = \cz{A}{H}$.
    Theorem~\ref{k:1} implies $\card{A:B} \leq r$.
    Observe that \[
        \cz{K}{A} \leq H \leq \cz{K}{B}.
    \]
    Suppose $A_{\infty} \leq \cz{A}{K}$.
    Lemma~\ref{aqs:5}(a) implies that $\cz{K}{A} \isomorphic K_{1}$
    whence $\cz{K}{A} = H$ and we are done.
    Suppose $A_{\infty} \not\leq \cz{A}{K}$.
    Now $H$ is overdiagonal and $H \leq \cz{K}{B}$
    so $\cz{K}{B}$ is overdiagonal.
    Lemma~\ref{aqs:5}(b) implies that $B \cap A_{\infty} \leq \cz{A}{K}$.
    As $\card{A:B} \leq r$ this forces $A = BA_{\infty}$ and then
    Lemma~\ref{aqs:5}(a) implies that $H = \cz{K}{B}$,
    again completing the proof in this case.

    Consider now the general case.
    Lemma~\ref{dp:3} implies there exists an $A$-invariant partition
    $\listset{\mathcal L_{1},\ldots,\mathcal L_{m}}$ of
    $\listset{K_{1},\ldots,K_{n}}$ such that \[
        H = (H \cap \listgen{\mathcal L_{1}}) \times\cdots\times%
        (H \cap \listgen{\mathcal L_{m}})
    \]
    and $H \cap \listgen{\mathcal L_{i}}$ is diagonal in
    $\listgen{\mathcal L_{i}}$ for each $i$.

    Let $A_{1} = %
    \kernel{(A\longrightarrow\sym{\listset{\mathcal L_{1},\ldots,\mathcal L_{m}}})}$.
    Since $A$ is abelian and transitive on $\listset{K_{1},\ldots,K_{n}}$
    it follows that $A_{1}$ is transitive on each $\mathcal L_{i}$.
    For each $i$, let $L_{i} = \listgen{\mathcal L_{i}}$,
    so $K = L_{1} \times\cdots\times L_{m}$ and we denote the projection map
    $K \longrightarrow L_{i}$ by $\lambda_{i}$.
    Lemma~\ref{p:6} implies that \[
        \cz{K}{A}\lambda_{i} = \cz{L_{i}}{A_{1}}.
    \]
    In particular, $H \cap L_{1}$ is $A_{1}\cz{L_{1}}{A_{1}}$-invariant.
    By the diagonal case,
    there exists $B \leq A_{1}$ with \[
        H \cap L_{1} = \cz{L_{1}}{B}.
    \]
    Now $A$ is abelian and transitive on $\listset{\mathcal L_{1},\ldots,\mathcal L_{m}}$
    whence $H \cap L_{i} = \cz{L_{i}}{B}$ for all $i$ and then \[
        H = \cz{K}{B}.
    \]
    Since $H$ is overdiagonal,
    Lemma~\ref{aqs:5}(b) implies that $B \cap A_{\infty} \leq \cz{A}{K}$.
\end{proof}

It remains to consider the $A\cz{K}{A}$-invariant underdiagonal subgroups.
Of particular interest is the case when there exist $A\cz{K}{A}$-invariant
solvable subgroups.
These are necessarily underdiagonal.

\begin{Lemma}\label{aqs:7} \mbox{}
    \begin{enumerate}
        \item[(a)]  For each $i$, $K_{i}$ possesses a unique maximal
                    $A_{\infty}\cz{K_{i}}{A_{\infty}}$-invariant
                    solvable subgroup $S_{i}$.
    \end{enumerate}
    Set $S = S_{1} *\cdots * S_{n}$.
    \begin{enumerate}
        \item[(b)]  $S$ is the unique maximal $A\cz{K}{A}$-invariant
                    solvable subgroup of $K$.

        \item[(c)]  Suppose $S \not\leq \zz{K}$.
                    Then $K$ is of type
                    $\ltwo{2^{r}}, \ltwo{3^{r}}, \uthree{2^{r}}$ or $\sz{2^{r}}$;
                    $\cz{K}{A} \leq \cz{K}{A_{\infty}} \leq S$ and
                    $S$ is a maximal $A$-invariant subgroup of $K$.
                    Moreover $S$ is the unique maximal $A\cz{K}{A}$-invariant
                    underdiagonal subgroup of $K$.
    \end{enumerate}
\end{Lemma}
\begin{proof}
    Using \CoprimeActionQuot\ and Lemma~\ref{p:5}
    we may suppose that $\zz{K} = 1$,
    so $K = K_{1} \times\cdots\times K_{n}$.
    For each $i$ let $\pi_{i}:K\longrightarrow K_{i}$ be the projection map.
    We may also assume $\cz{A}{K} = 1$,
    so Lemma~\ref{aqs:4} implies $\card{A_{\infty}} = 1$ or $r$
    and $A_{\infty}$ acts faithfully on each $K_{i}$.

    (a). If $A_{\infty} = 1$ then $K_{i} = \cz{K_{i}}{A_{\infty}}$
    and $K_{i}$ is simple so put $S_{i} = 1$.
    If $\card{A_{\infty}} = r$ then the existence of $S_{i}$
    follows from Theorem~\ref{k:1}(c).

    (b). Since $A_{\infty} \normal A$ it follows that $A$ permutes transitively
    the subgroups $A_{\infty}\cz{K_{i}}{A}$ and then that $A$ permutes
    the subgroups $S_{i}$.
    Thus $S$ is an $A$-invariant solvable subgroup of $K$.
    Lemma~\ref{p:6} implies $\cz{K}{A}\pi_{i} = \cz{K_{i}}{A_{\infty}}$
    and it follows that $S$ is $A\cz{K}{A}$-invariant.

    Suppose $H$ is an $A\cz{K}{A}$-invariant solvable subgroup of $K$.
    Now $H \leq H\pi_{1} \times\cdots\times H\pi_{n}$ and as
    $\cz{K}{A}\pi_{i} = \cz{K_{i}}{A_{\infty}}$ it follows that each
    $H\pi_{i}$ is an $A_{\infty}\cz{K_{i}}{A}$-invariant
    solvable subgroup of $K_{i}$.
    Then $H\pi_{i} \leq S_{i}$ and $H \leq S$.

    (c). Apply Theorem~\ref{k:1}(c).
\end{proof}

\begin{Lemma}\label{aqs:8} \mbox{}
    \begin{enumerate}
        \item[(a)]  For each $i$, $K_{i}$ possesses a unique maximal
                    $A_{\infty}\cz{K_{i}}{A_{\infty}}$-invariant
                    proper subgroup $M_{i}$.
    \end{enumerate}
    Set $M = M_{1} *\cdots * M_{n}$.
    \begin{enumerate}
        \item[(b)]  $M$ is the unique maximal $A\cz{K}{A}$-invariant
                    underdiagonal subgroup of $K$.

        \item[(c)]  Suppose $M \not\leq \zz{K}$.
                    Then $A_{\infty} \not\leq \cz{A}{K}$
                    and $\cz{K}{A_{\infty}} \leq M$.
                    If in addition $M \not= \cz{K}{A_{\infty}}\zz{K}$
                    then $K$ is of type $\ltwo{2^{r}}$ or $\sz{2^{r}}$
                    and $M$ is solvable.
    \end{enumerate}
\end{Lemma}
\begin{proof}
    The proof is similar to the proof of Lemma~\ref{aqs:7}
    but using Theorem~\ref{k:1}(d) in place of Theorem~\ref{k:1}(c).
\end{proof}

\begin{Corollary}\label{aqs:9}
    Let $a \in A\nonid$ and suppose $H$ is an $A$-invariant subgroup
    that satisfies $\cz{K}{a} \leq H \leq K$
    and $H^{(\infty)} \not\leq \cz{K}{a}$.
    Then $H = K$.
\end{Corollary}
\begin{proof}
    Using \CoprimeActionQuot\ we may assume that $\zz{K} = 1$.
    We may also assume that $\cz{A}{K} = 1$.
    Suppose that $H$ is underdiagonal.
    Lemma~\ref{aqs:8} implies that $H \leq \cz{K}{A_{\infty}}$.
    Now $\cz{K}{a} < H$ so $\cz{K}{a}$ is also underdiagonal,
    whence $a \in A_{\infty}$.
    As $\card{A_{\infty}} \leq r$ we obtain $\listgen{a} = A_{\infty}$
    whence $H = \cz{K}{a}$,
    a contradiction.
    We deduce that $H$ is overdiagonal.

    Lemma~\ref{aqs:6} implies $H = \cz{K}{B}$ for some $B \leq A$.
    Using Lemma~\ref{aqs:5}(d) we have $B \leq \cz{A}{\cz{K}{a}} = \listgen{a}$
    whence $B = 1$ or $\listgen{a}$.
    Now $\cz{K}{a} < H = \cz{K}{B}$ whence $B = 1$ and $H = K$.
\end{proof}

\begin{Lemma}\label{aqs:10}
    Suppose that $H$ is an $A\cz{K}{A}$-invariant subgroup of $K$
    and that $L \in \comp{A}{H}$.
    Then $L = \layerr{H}$ and either
    \begin{enumerate}
        \item[(a)]  $\cz{K}{A} = \cz{L}{A}$ and $L$ is overdiagonal; or

        \item[(b)]  $\layerr{\cz{K}{A}} = \layerr{\cz{L}{A}} \not= 1$
                    and $L$ is underdiagonal.
    \end{enumerate}
\end{Lemma}
\begin{proof}
    Lemmas~\ref{aqs:5}, \ref{aqs:6} and \ref{aqs:8} imply that $\layerr{H}$
    is either trivial or $A$-quasisimple.
    Since $L \in \comp{A}{H}$ it follows that $L = \layerr{H}$.
    In particular, $L$ is $A\cz{K}{A}$-invariant.

    Suppose that $L$ is overdiagonal.
    Lemma~\ref{aqs:6} implies that $L = \cz{K}{B}(L \cap \zz{K})$
    for some $B \leq A$ with $B \cap A_{\infty} \leq \cz{A}{K}$.
    Lemma~\ref{aqs:5} implies that $\cz{K}{B}$ is $A$-quasisimple.
    Since $L$ is also $A$-quasisimple,
    it follows that $L = \cz{K}{B}$.
    Now $B \leq A$ whence $\cz{L}{A} = \cz{K}{A}$ and (a) holds.
    Hence we may assume that $L$ is underdiagonal.

    Let $\br{K} = K/\zz{K}$.
    Suppose $\cz{\br{K}}{A}$ is solvable.
    Lemma~\ref{aqs:7} implies that $K$ is of type
    $\ltwo{2^{r}}, \ltwo{3^{r}}, \uthree{2^{r}}$ or $\sz{2^{r}}$.
    Lemma~\ref{aqs:7}(c) implies that any $A\cz{K}{A}$-invariant
    proper subgroup of $\br{K}$ is solvable.
    Then $\br{L} = \br{K}$ contrary to $L$ being underdiagonal.
    Hence $\cz{\br{K}}{A}$ is nonsolvable.
    Lemma~\ref{aqs:5} implies that $\gfitt{\cz{\br{K}}{A}}$ is simple.
    Now $\gfitt{\cz{\br{L}}{A}} \normal \gfitt{\cz{\br{K}}{A}}$ whence
    $\layerr{\cz{\br{L}}{A}} = \layerr{\cz{\br{K}}{A}} \not= 1$.
    \CoprimeActionQuot\ implies that
    $\layerr{\br{\cz{L}{A}}} = \layerr{\br{\cz{K}{A}}}$.
    Since $\br{K} = K/\zz{K}$ it follows that
    $\layerr{\cz{L}{A}}\zz{K} = \layerr{\cz{K}{A}}\zz{K}$.
    Then (b) follows on taking the derived subgroup of both sides.
\end{proof}

We record the following triviality.

\begin{Lemma}\label{aqs:11}
    Let $a \in A$.
    Then $[K,a] = 1$ or $K$.
\end{Lemma}
\begin{proof}
    Suppose $[K,a] \not= 1$.
    Set $\br{K} = K/\zz{K}$.
    Lemma~\ref{p:5} implies that $[\br{K},a] \not= 1$.
    Now $A$ is abelian so $[\br{K},a]$ is an $A$-invariant normal
    subgroup of $\br{K}$.
    Then $\br{K} = [\br{K},a]$ because $\br{K}$ is $A$-simple.
    Consequently $K = [K,a]\zz{K}$ so as $K$ is perfect,
    we obtain $K = [K,a]$.
\end{proof}

We close with a lemma of generation.
Recall that $\hyp{A}$ denotes the set of subgroups of index $r$ in $A$.

\begin{Lemma}\label{aqs:12} \mbox{}
    \begin{enumerate}
        \item[(a)]  Let $B \leq A$ and suppose $\cz{K}{B}$ is overdiagonal.
                    Then $\cz{K}{C}$ is overdiagonal and $A$-quasisimple
                    for all $C \leq B$.

        \item[(b)]  Suppose $1 \not= A^{*} \leq A$.
                    Then \[
                        K = \gen{\cz{K}{B}}{\mbox{$B \in \hyp{A^{*}}$ and $\cz{K}{B}$ is overdiagonal}}.
                    \]
    \end{enumerate}
\end{Lemma}
\begin{proof}
    (a). Lemma~\ref{aqs:5}(b) implies $B \cap A_{\infty} \leq \cz{A}{K}$.
    Then for each $C \leq B$ we have $C \cap A_{\infty} \leq \cz{A}{K}$.
    The conclusion follows from Lemma~\ref{aqs:5}(a).

    (b). If $A^{*} \leq \cz{A}{K}$ then $K = \cz{K}{B}$ for any $B \in \hyp{A^{*}}$.
    Hence we may assume that $A^{*} \not\leq \cz{A}{K}$.
    Then $A^{*}$ has nontrivial image in $A/\cz{A}{K}$ and we may replace
    $A$ by $A/\cz{A}{K}$ to assume that $\cz{A}{K} = 1$.

    Let $\mathcal H$ be the set of hyperplanes of $A^{*}$
    that intersect $A_{\infty}$ trivially.
    Lemma~\ref{aqs:4} implies $\card{A^{*} \cap A_{\infty}} = 1$ or $r$.
    Note that if $A^{*} \cap A_{\infty} = A^{*}$ then $\card{A^{*}} = r$
    and $\mathcal H = \listset{1}$.
    It follows that $\cap\mathcal H = 1$.
    Let \[
        L = \gen{ \cz{K}{B} }{ B \in \mathcal H }.
    \]
    Now $\cz{K}{B}$ is overdiagonal and perfect for each $B \in \mathcal H$
    by Lemma~\ref{aqs:5}(a).
    It follows that $L$ is overdiagonal and perfect.
    Lemma~\ref{aqs:6} implies $L = \cz{K}{C}$ for some $C \leq A$.
    Then, using Lemma~\ref{aqs:5}(d), we have \[
        C \leq \cz{A}{L} = \bigcap_{B \in \mathcal H}\cz{A}{\cz{K}{B}} %
        = \bigcap_{B \in \mathcal H}B = 1.
    \]
    Then $L = \cz{K}{C} = K$,
    completing the proof.
\end{proof}
\section{Modules}\label{m}
Two results on modules,
which are central to the theory being developed in this paper,
will be established.
The first result has previously been proved by the author \cite{F4}.
The proof presented here is much shorter.
However, it requires the $K$-group hypothesis whereas
the proof in \cite{F4} does not.

\begin{Theorem}\label{m:1}
    Let $R$ be a group of prime order $r$ that acts on the $r'$-group $G$.
    Assume that $G$ is a $K$-group.
    Let $V$ be a faithful completely reducible $\pring{F}{RG}$-module
    over a field $\field{F}$ of characteristic $p$.
    Assume that $\pring{F}{R}$ is not a direct summand of $V_{R}$.
    Then either:
    \begin{itemize}
        \item   $[G,R] = 1$ or

        \item   $r$ is a Fermat prime and $[G,R]$ is a special $2$-group.
    \end{itemize}
\end{Theorem}
\begin{proof}
    Assume false and let $G$ be a minimal counterexample.
    By \cite[Theorem~5.1]{F1}, $G$ is nonsolvable.
    Now $R[G,R] \normal RG$ so $V_{R[G,R]}$ is completely reducible
    by Clifford's Theorem.
    Then \CoprimeActionComm\ and the minimality of $G$ imply $G = [G,R]$.
    If $p=0$ then define $\op{p}{H} = 1$ for any group $H$.
    Since $V$ is completely reducible we have $\op{p}{G} = 1$.

    \setcounter{Claim}{0}
    \begin{Claim}\label{m:1:Claim1}
        Let $H$ be a proper $R$-invariant subgroup of $G$.
        \begin{enumerate}
            \item[(a)]  Suppose $H = [H,R]$.
                        Then $H/\op{p}{H}$ is either trivial or a nonabelian $2$-group.

            \item[(b)]  Suppose $H$ is a $q$-group for some prime $q \not= p$.
                        If $q=2$ assume $H$ is abelian.
                        Then $[H,R] = 1$.
        \end{enumerate}
    \end{Claim}
    \begin{proof}
        (a). Let $U$ be an irreducible constituent of $V_{RH}$.
        Irreducibility implies $\op{p}{H/\cz{H}{U}} = 1$ and then the minimality
        of $G$ implies $H/\cz{H}{U}$ is either trivial or a nonabelian $2$-group.
        Lemma~\ref{p:9}(b) implies $\cap\cz{H}{U} = \op{p}{H}$ where $U$
        ranges over the irreducible constituents of $V_{RH}$,
        so the result follows.

        (b). This follows from \CoprimeActionComm\ and (a).
    \end{proof}

    \begin{Claim}\label{m:1:Claim2}
        $\ff{G} \leq \zz{G} \leq \cz{G}{R}$.
    \end{Claim}
    \begin{proof}
        Complete reducibility implies $\op{p}{G} = 1$.
        Then Claim~\ref{m:1:Claim1}(b) implies $\zz{G} \leq \cz{G}{R}$.
        Assume that $\ff{G} \not\leq \zz{G}$.
        Since $G = [G,R]$ it follows that $[\ff{G},R] \not= 1$.
        Claim~\ref{m:1:Claim1}(b) implies $[\op{q}{G},R] = 1$
        for each prime $q \not\in \listset{2,p}$.
        As $\op{p}{G} = 1$ we deduce that $[\op{2}{G},R] \not= 1$ and that $p \not= 2$.

        Let $C = \cz{G}{\op{2}{G}}$.
        Now $[C,R] \normal C \normal G$ so
        $\op{p}{[C,R]} \leq \op{p}{C} \leq \op{p}{G} = 1$.
        Claim~\ref{m:1:Claim1}(a) implies $[C,R]$ is a $2$-group,
        whence $[C,R] \leq \op{2}{C} \leq \op{2}{G}$.
        As $C = \cz{G}{\op{2}{G}}$ we obtain $[C,R] \leq \zz{\op{2}{G}}$.
        Using \CoprimeActionComm\ and Claim~\ref{m:1:Claim1}(b) we have
        $[C,R] = [C,R,R] \leq [\zz{\op{2}{G}},R] = 1$.
        As $C \normal RG$ and $G = [G,R]$ we deduce that $C \leq \zz{G}$.

        Let $t\not=2$ be a prime.
        By \CoprimeActionSyl\ there exists an
        $R$-invariant Sylow $t$-subgroup $T$ of $G$.
        Set $H = T\op{2}{G}$ and $H_{0} = [H,R] \normal H$.
        Then $H$ is solvable so $H \not= G$.
        Now $\op{p}{H_{0}} \leq \op{p}{H} \leq \cz{G}{\op{2}{G}} \leq \zz{G}$
        so as $\op{p}{G} = 1$ we deduce that $\op{p}{H_{0}} = 1$.
        Claim~\ref{m:1:Claim1}(a) implies that $H_{0}$ is a $2$-group.
        Since $[T,R] \leq H_{0}$ and $t\not=2$ we deduce that $\cz{G}{R}$
        contains a Sylow $t$-subgroup of $G$ for each prime $t\not=2$.

        Let $S$ be an $R$-invariant Sylow $2$-subgroup of $G$.
        The previous paragraph implies $G = \cz{G}{R}S$.
        Then $G = [G,R] \leq S$,
        contrary to $G$ being nonsolvable.
        We deduce that $\ff{G} \leq \zz{G}$.
    \end{proof}

    \CoprimeActionGFitt\ implies $[\gfitt{G},R] \not= 1$ so as
    $[\ff{G},R] = 1$ there exists $K \in \compp{G}$ with $[K,R] \not= 1$.

    \begin{Claim}\label{m:1:Claim3}
        $R$ normalizes $K$.
    \end{Claim}
    \begin{proof}
        Assume false.
        Let $K_{1},\ldots,K_{r}$ be the $R$-conjugates of $K$.
        Define $L = \listgen{K_{1},\ldots,K_{r}}$ and $\br{L} = L/\zz{L}$.
        Then $L = K_{1}*\cdots *K_{r}$ and $\br{L} = \br{K_{1}}\times\cdots\times\br{K_{r}}$.
        Choose $q \in \primes{\br{K_{1}}}$ with $q\not=p$
        and let $\br{Q_{1}} \leq \br{K_{1}}$ have order $q$.
        Let $Q_{1} \leq K_{1}$ be a $q$-subgroup that maps onto $\br{Q_{1}}$.
        Then $Q_{1}$ is abelian.
        Let $Q_{1},\ldots,Q_{r}$ be the $R$-conjugates of $Q_{1}$
        and set $Q = Q_{1}*\cdots * Q_{r}$.
        Then $[Q,R] \not= 1$ since $R$ does not normalize $\br{Q_{1}}$.
        But $Q$ is abelian so Claim~\ref{m:1:Claim1}(b) implies $[Q,R] = 1$,
        a contradiction.
        The claim is established.
    \end{proof}

    Now $K$ is quasisimple and $[K,R] \not= 1$ so $K = [K,R]$.
    Claim~\ref{m:1:Claim1}(b) forces $G = K$.
    Theorem~\ref{k:1}(f) supplies a contradiction.
\end{proof}

The next result is a partial extension of the main result of
\cite{F1} to nonsolvable groups.

\begin{Theorem}\label{m:2}
    Let $R$ be a group of prime order $r$ that acts coprimely
    on the $K$-group $G$.
    Let $V$ be an $RG$-module, possibly of mixed characteristic,
    with $V_{[G,R]}$ faithful and completely reducible.
    Suppose that \[
        K \in \compp{\kernelon{\cz{G}{R}}{\cz{V}{R}}}.
    \]
    Then $K \in \compp{G}$.
\end{Theorem}
\begin{proof}
    Let $F = \ff{G}$ and $M = KF$.
    Now $K \in \compp{\cz{G}{R}}$ whence $[K,\cz{F}{R}] = 1$
    and so $K \normal \cz{M}{R}$.
    Also $[M,R] = [F,R] \normal F \cap [G,R] \normal [G,R]$
    so as $V_{[G,R]}$ is completely reducible,
    Clifford's Theorem implies that $V_{[M,R]}$ is also.

    Let $L$ be the subnormal closure of $K$ in $M$.
    Then $L = \listgen{K^{L}}$.
    Now $[M,R]$ is solvable so \cite[Theorem~A]{F1} implies that
    $L = K(S \times P)$ with $S$ a $2$-group; $S = [S,R]$;
    $S' = \cz{S}{R}$; $\cz{K}{S'} = \cz{K}{S}$;
    $P$ a $p$-group for some odd prime $p$ and $K/\cz{K}{P}$ a $2$-group.
    Since $S = [S,R] \leq [M,R] \leq F$ and $S' = \cz{S}{R}$
    we have $[K,S'] = 1$.
    Then $[K,S] = 1$.
    Also, $K$ is perfect so as $K/\cz{K}{P}$ is a $2$-group
    it follows that $K = \cz{K}{P}$.
    Then $K \normal L = \listgen{K^{L}}$ whence $K = L \subnormal M$
    and $K$ is a component of $M$.
    Since $M = KF$ and $F$ is nilpotent,
    we obtain $[K,F] = 1$.

    Since $\cz{G}{\gfitt{G}} = \zz{\ff{G}}$,
    there exists $X \in \comp{R}{G}$ with $[K,X] \not= 1$.
    Now $\cz{X}{R} \subnormal \cz{G}{R}$ and $K \in \compp{\cz{G}{R}}$ so
    $K \leq \cz{X}{R}$ or $[K,\cz{X}{R}] = 1$ by Lemma~\ref{p:2}(a).
    Theorem~\ref{k:4}(c) rules out the second possibility,
    whence $K \leq \cz{X}{R}$.
    In particular, $K \in \compp{\cz{X}{R}}$.
    If $X = \cz{X}{R}$ then $K \subnormal X \subnormal G$
    whence $K \in \compp{G}$.
    Hence we may assume, for a contradiction, that $[X,R] \not= 1$.
    Lemma~\ref{aqs:11} implies that $X = [X,R]$.

    We have $K \leq X = [X,R] \subnormal [G,R]$.
    Clifford's Theorem implies that $V_{X}$ is completely reducible.
    Hence we may assume that $G = X$,
    so $G$ is $R$-quasisimple and $G = [G,R]$.
    In particular, $V_{G}$ is completely reducible and so \[
        V = \cz{V}{G} \oplus [V,G].
    \]

    Let $U$ be an irreducible $RG$-submodule contained in $[V,G]$.
    Now $G$ is $R$-quasisimple so either
    $\cz{G}{U} \leq \zz{G}$ or $\cz{G}{U} = G$.
    The second possibility does not hold since $\cz{V}{G} \cap [V,G] = 0$.
    Thus $\cz{G}{U} \leq \zz{G}$.
    Set $\br{G} = G/\cz{G}{U}$.
    Then $\br{G}$ is $R$-quasisimple,
    $\br{G} = [\br{G},R]$ and $\br{K} \not= 1$.
    Suppose that $U \not= V$.
    By induction, $\br{K} \in \compp{\br{G}}$,
    whence $\br{K} = \br{G}$.
    This is a contradiction since $[\br{K},R] = 1$
    but $[\br{G},R] = \br{G}$.
    We deduce that $V$ is an irreducible $\pring{F}{RG}$-module
    for some field $\field{F}$.

    Theorem~\ref{k:1}(f) implies that $G$ is not quasisimple.
    The remainder of the argument is an extension of Theorem~\ref{k:1}(f)
    to $R$-quasisimple groups that are not quasisimple.
    We remark that no $K$-group hypothesis is required.

    We have $G = K_{1} * \cdots * K_{r}$ where $K_{1},\ldots,K_{r}$
    are quasisimple subgroups that are permuted transitively by $R$.
    Lemma~\ref{aqs:5}(a) implies that $\cz{G}{R}$ is quasisimple.
    Since $K \in \compp{\cz{G}{R}}$ we deduce that \[
        K = \cz{G}{R}
    \]
    and then that $[\cz{V}{R},\cz{G}{R}] = 0$.

    By Burnside's $p^{\alpha}q^{\beta}$-Theorem,
    we may choose $t \in \primes{K_{1}}$ with $t \not\in \listset{2, \fieldchar{F}}$.
    Now $K_{1}$, being quasisimple, is not $t$-nilpotent
    so Frobenius' Normal Complement Theorem implies there exists
    a $t$-subgroup $T_{1} \leq K_{1}$ and a
    $t'$-subgroup $S_{1} \leq \n{K_{1}}{T_{1}}$
    with $1 \not= T_{1} = [T_{1},S_{1}]$.
    Set $T = \listgen{T_{1}^{R}} = T_{1} * \cdots * T_{r}$
    and $S = \listgen{S_{1}^{R}} = S_{1} * \cdots * S_{r}$
    where $T_{1},\ldots,T_{r}$ and $S_{1},\ldots,S_{r}$
    are the conjugates of $T_{1}$ and $S_{1}$ under the action of $R$.

    Considering $\br{G} = G/\zz{G}$,
    we see that $S_{1} \leq \cz{S}{R}(S_{2} * \cdots * S_{r})\zz{K}$
    whence $T_{1} = [T_{1},S_{1}] = [T_{1},\cz{S}{R}]$.
    It follows that $T = [T,\cz{S}{R}]$.
    Now $[\cz{V}{R},\cz{G}{R}] = 0$,
    so we may apply Lemma~\ref{p:11},
    with $\cz{S}{R}$ in the role of $S$,
    to deduce that \[
        [V,[T,R]] = 0.
    \]
    But then $[T,R] = 1$,
    a contradiction since $T_{1} \not\leq \zz{K_{1}}$
    and so $T_{1}$ is not normalized by $R$.
    The proof is complete.
\end{proof}
\section{General Results}\label{g}
The first result is the starting point for the study
of how the global structure of a group that admits
a group of automorphisms is influenced by its
local structure.
The other results are applications of the module results
from \S\ref{m} to composite groups.

\begin{Lemma}\label{g:1}
    Let $A$ be an elementary abelian $r$-group for some prime $r$ that
    acts coprimely on the $K$-group $G$.
    Suppose that $H$ is an $A\cz{G}{A}$-invariant subgroup
    of $G$ and that $K \in \comp{A}{H}$.
    Then there exists a unique $\widetilde{K}$ with \[
        K \leq \widetilde{K} \in \compasol{G}.
    \]
\end{Lemma}
\begin{proof}
    We may suppose that $\cz{G}{A} \leq H$.
    Uniqueness is clear since distinct $(A,\mbox{sol})$-components
    have solvable intersection.
    Using \CoprimeActionQuot\ and the correspondence between
    $(A,\mbox{sol})$-components of $G$ and $A$-components of $G/\sol{G}$,
    it suffices to assume $\sol{G} = 1$ and show that $K$ is contained
    in an $A$-component of $G$.

    Since $\sol{G} = 1$ we have $\cz{G}{\layerr{G}} = 1$ so there
    exists $L \in \comp{A}{G}$ with $[L,K] \not= 1$.
    Now $\cz{L}{A} \leq L \cap H \subnormal H$ and $K \in \comp{A}{H}$.
    Since $L \cap H$ is $A$-invariant it follows from Lemma~\ref{p:2}(a)
    that either $K \leq L \cap H$ or $[K,L \cap H] = 1$.
    As $[L,K] \not= 1$,
    Theorem~\ref{k:4}(c) rules out the second possibility.
    Thus $K \leq L$, completing the proof.
\end{proof}

We remark that it is straightforward to construct examples
where $\widetilde{K}$ is not an $A$-component.

\begin{Lemma}\label{g:2}
    Let $R$ be a group of prime order $r$ that acts coprimely on
    the group $G$.
    Suppose that $K$ and $S$ are $R$-invariant subgroups of $G$
    that satisfy:
    \begin{itemize}
        \item   $K = [K,R]$ and $K$ is a $K$-group.
        \item   $K = \oupper{2}{K}$ or $r$ is not a Fermat prime.
        \item   $S$ is $K$-invariant and solvable.
        \item   $KS = \listgen{K^{S}}$.
    \end{itemize}
    Then \[
        KS = \listgen{K,\cz{S}{R}}.
    \]
\end{Lemma}
\begin{proof}
    We may assume that $G = KS$, so $S \normal G$.
    Let $V$ be a minimal $R$-invariant normal subgroup of $G$
    contained in $S$.
    Then $V$ is an elementary abelian $p$-group for some prime $p$
    and hence an irreducible $\gfpring{p}{RG}$-module.
    By induction and \CoprimeActionQuot\ we obtain \[
        G = V\listgen{K,\cz{S}{R}}.
    \]
    Suppose that $V \cap \listgen{K,\cz{S}{R}} \not= 1$.
    The choice of $V$ forces $V \leq \listgen{K,\cz{S}{R}}$,
    whence $G = \listgen{K,\cz{S}{R}}$.
    Hence we may suppose that $V \cap \listgen{K,\cz{S}{R}} = 1$.
    Now $V \leq S$ whence $\cz{V}{R} = 1$.

    Set $\br{G} = G/\cz{G}{V}$.
    Now $K = [K,R] \leq [G,R]  \normal G$ so as $G = KS = \listgen{K^{S}}$
    we have $G = [G,R]$ and then $\br{G} = [\br{G},R]$.
    Similarly, if $r$ is a Fermat prime then as $K = \oupper{2}{K}$ we have $G = \oupper{2}{G}$
    and $\br{G} = \oupper{2}{\br{G}}$.
    Theorem~\ref{m:1} implies $\br{G} = 1$.
    Hence $V \leq \zz{G}$ so  \[
        G = \listgen{K^{S}} = \listgen{K^{G}} \leq \listgen{K, \cz{S}{R}}
    \]
    and the proof is complete.
\end{proof}

\begin{Lemma}\label{g:3}
    Let $R$ be a group of prime order $r$ that acts coprimely
    on the group $G$.
    Suppose that $K$ and $S$ are $R$-invariant subgroups of $G$
    that satisfy:
    \begin{itemize}
        \item   $K = [K,R]$ and $K$ is a $K$-group.
        \item   $K$ is perfect.
        \item   $S$ is a $K$-invariant solvable subgroup.
        \item   $\cz{S}{R} \leq \n{G}{K}$.
    \end{itemize}
    Then \[
        S \leq \n{G}{K}.
    \]
    If in addition $\sol{K} = \zz{K}$ then $[S,K] = 1$.
\end{Lemma}
\begin{proof}
    We may assume that $G = KS$, so $S \normal G$.
    Let $H$ be the smallest subnormal subgroup of $G$ that contains $K$.
    Then $H$ is $R$-invariant and
    $H = K(H \cap S) = \listgen{K^{H}} = \listgen{K^{H \cap S}}$.
    Now $K$ is perfect so $K = \oupper{2}{K}$.
    Lemma~\ref{g:2} implies that $H = \listgen{K, \cz{H\cap S}{R}}$.
    Since $\cz{S}{R} \leq \n{G}{K}$ we obtain $K \normal H$
    and then $K = H$,
    so $K \subnormal G$.
    The conclusion follows from Lemma~\ref{p:2}(b).
\end{proof}

\begin{Lemma}\label{g:4}
    Let $R$ be a group of prime order $r$ that acts coprimely
    on the $K$-group $G$.
    Suppose that $G$ is constrained and that $K \in \compp{\cz{G}{R}}$.
    Set $\br{G} = G/\ff{G}$.
    Then $\br{K} \in \compp{\br{G}}$.
    In particular, $K\ff{G} \subnormal G$ and $[K,\sol{G}] \leq \ff{G}$.
\end{Lemma}
\begin{proof}
    Let $G_{0}$ be the smallest subnormal subgroup of $G$ that contains $K$.
    Note that every subnormal subgroup of a constrained group is constrained.
    Then $G_{0}$ is $R$-invariant and constrained.
    Suppose the result has been established for $G_{0}$.
    Then $K\ff{G_{0}} \subnormal G_{0} \normal G$ whence $K\ff{G_{0}} \subnormal G$.
    Now $\ff{G_{0}} \subnormal G$ so $\ff{G_{0}} \leq \ff{G}$
    whence $K\ff{G} \subnormal G$ and the conclusion follows.
    Hence we may assume that $G = G_{0}$.
    In particular, $G = \listgen{K^{G}}$.

    Since $K \in \compp{\cz{G}{R}}$ we have $[K,\cz{\ff{G}}{R}] = 1$.
    Let $V$ be a chief factor of $RG$ contained in $\ff{G}$.
    Then $V$ is an elementary abelian $p$-group for some prime $p$.
    Set $G^{*} = G/\cz{G}{V}$,
    so $V$ is a $\gfpring{p}{RG^{*}}$-module.
    Now $K \in \compp{\cz{G}{R}}$ and $\cz{V}{R} \leq \ff{\cz{G}{R}}$
    so $[K,\cz{V}{R}] = 1$.
    \CoprimeActionQuot\ implies that either \[
        K^{*} = 1 \mbox{ or } K^{*} \in \compp{\kernelon{\cz{G^{*}}{R}}{\cz{V}{R}}}.
    \]
    In the first case, as $G = \listgen{K^{G}}$, we have $G^{*} = 1$.
    In the second case,
    Theorem~\ref{m:2} implies $K^{*} \in \compp{G^{*}}$.
    As $G = \listgen{K^{G}}$ this implies $G^{*} = K^{*}$.
    In particular, $[G^{*},R] = 1$.

    We have shown that \[
        [G,R] \leq \bigcap\cz{G}{V}
    \]
    where $V$ ranges over the chief factors of $RG$ contained in $\ff{G}$.
    Lemma~\ref{p:12} implies that $[G,R] \leq \ff{G}$.
    By \CoprimeActionComm\ we have $G = \cz{G}{R}[G,R]$ so as $K \in \compp{\cz{G}{R}}$
    it follows that $K\ff{G} \subnormal G$.
    This completes the proof.
\end{proof}

\section{Local to global results}\label{l}
\begin{Theorem}\label{l:1}
    Let $r$ be a prime and $A$ an elementary abelian $r$-group
    that acts coprimely on the $K$-group $G$.
    Let $a \in A\nonid$ and let $H$ be an $A\cz{G}{a}$-invariant
    subgroup of $G$.
    Suppose that $K \in \comp{A}{H}$.
    \begin{enumerate}
        \item[(a)]  There exists a unique $\widetilde{K}$ with
                    $K \leq \widetilde{K} \in \compasol{G}$.

        \item[(b)]  If $[K,a] = 1$ then $K = \layerr{\cz{\widetilde{K}}{a}}$.

        \item[(c)]  If $[K,a] \not= 1$ then $K = [K,a] = \widetilde{K}$.
                    In particular, $K \in \comp{A}{G}$.

        \item[(d)]  If $\widetilde{K}$ is constrained then
                    $\widetilde{K} = K\ff{\widetilde{K}}$ and $[K,a] = 1$.
                    In particular, $K$ is an $A$-component of $G$ modulo $\ff{G}$.

        \item[(e)]  Suppose $L \in \compasol{G}$ with $\widetilde{K} \not= L$
                    and $L = [L,a]$.
                    Then $[\widetilde{K},L] = 1$.
    \end{enumerate}
\end{Theorem}

\noindent Before launching into the proof,
a number of remarks are in order.
Firstly, an important special case is when $A = \listgen{a}$
and $H = \cz{G}{A}$.
Secondly, there are of course two quite different outcomes.
Either $\widetilde{K}$ is semisimple or constrained.
In some senses,
the first outcome is the most desired --
the $A$-component $K$ of $H$ is contained in
the $A$-component $\widetilde{K}$ of $G$.
What part~(d) shows is that in the constrained case,
the situation is quite controlled.
Thirdly, turning to part~(e),
recall that distinct $A$-components of $G$ commute.
This fact plays a crucial role in many arguments.
Although distinct $(A,\mbox{sol})$-components normalize each other,
they do not necessarily commute.
Part~(e) removes the need to be concerned about this phenomena.

\begin{proof}
    Set $R = \listgen{a}$.
    Now $K$ is $R$-invariant so it follows from commutator identities
    that $[K,R] = [K,a]$.
    Also, $K \subnormal H \normal H\cz{G}{a}$ so $K \in \comp{A}{H\cz{G}{a}}$.
    Hence we may assume that $\cz{G}{a} \leq H$.

    (a). This follows from Lemma~\ref{g:1}.

    (b). Suppose $[K,a] = 1$.
    Now $K \in \comp{A}{H}$ so $K \in \comp{A}{\cz{G}{a}}$.
    Then $K \in \comp{A}{\cz{\widetilde{K}}{a}}$.
    Let $\widetilde{K}^{*} = \widetilde{K}/\sol{\widetilde{K}}$,
    so $\widetilde{K}^{*}$ is $A$-simple.
    Lemma~\ref{aqs:5} implies that $\cz{\widetilde{K}}{a}$
    has at most one $A$-component,
    whence $K = \layerr{\cz{\widetilde{K}}{a}}$.

    (c). Since $[K,a] \not= 1$,
    Lemma~\ref{aqs:11} implies $K = [K,a]$.
    Let $S = \sol{\widetilde{K}}$.
    Now $S \cap H$ is a solvable normal subgroup of $H$
    and $K \in \comp{A}{H}$ so $[K, S \cap H] = 1$.
    In particular, $[K,\cz{S}{a}] = 1$.
    Lemma~\ref{g:3} forces $[K,S] = 1$.
    Consequently
    $\cz{\widetilde{K}}{\ff{\widetilde{K}}} \not\leq \ff{\widetilde{K}}$
    so $\widetilde{K}$ is not constrained.
    Since $\widetilde{K} \in \compasol{G}$ it follows that
    $\widetilde{K}$ is $A$-quasisimple.
    Now $\cz{\widetilde{K}}{a} \leq H \cap \widetilde{K}$.
    Moreover, $K^{(\infty)} = K = [K,a] \leq H$ so
    Corollary~\ref{aqs:9} forces $H \cap \widetilde{K} = \widetilde{K}$,
    whence $\widetilde{K} \leq H$.
    Now $K \leq \widetilde{K}$ and $K \in \comp{A}{H}$ so
    $K \in \comp{A}{\widetilde{K}}$ and then $K = \widetilde{K}$
    since $\widetilde{K}$ is $A$-quasisimple.

    (d). Since $\widetilde{K}$ is constrained it is not equal to $K$
    so (c) implies $[K,a] = 1$.
    Then $K \in \comp{A}{\cz{G}{a}}$ and so $K \in \comp{A}{\cz{\widetilde{K}}{a}}$.
    Since $K$ is $A$-quasisimple,
    it is the central product of its components.
    Let $K_{0} \in \compp{K}$.
    Then $K_{0} \in \compp{\cz{\widetilde{K}}{a}}$.
    Lemma~\ref{g:4} implies $K_{0}\ff{\widetilde{K}} \subnormal \widetilde{K}$.
    It follows that $K\ff{\widetilde{K}} \subnormal \widetilde{K}$.
    Now $\widetilde{K}$ is minimal subject to being $A$-invariant,
    nonsolvable and subnormal in $G$ so this forces
    $K\ff{\widetilde{K}} = \widetilde{K}$.
    Finally, $\ff{\widetilde{K}} \leq \ff{G}$ whence $K$ is a
    component of $G$ modulo $\ff{G}$.

    (e). Note that $\widetilde{K}$ and $L$ normalize each other
    and that $[\widetilde{K},L] \leq \sol{G}$ since $\widetilde{K} \not= L$.
    We may assume that $G = \widetilde{K}L$ and that $\sol{G} \not= 1$.
    Let $V$ be a minimal $A$-invariant normal subgroup of $G$ that is
    contained in $\sol{G}$.
    Then $V$ is an elementary abelian $p$-group for some prime $p$.
    Let $G^{*} = G/V$.
    Note that $\widetilde{K}^{*}$ and $L^{*}$ are distinct since their
    commutator is solvable.
    Using \CoprimeActionQuot\ and induction,
    we conclude that $[\widetilde{K}^{*},L^{*}] = 1$.
    Then \[
        [\widetilde{K},L] \leq V.
    \]

    Since $K \in \comp{A}{H}$ and $V \cap H$ is a solvable normal subgroup of $H$
    we have $[K, V \cap H] = 1$.
    In particular $[K, \cz{V}{a}] = 1$,
    so $\cz{V}{a} \leq \cz{V}{K}$.
    Now $[\widetilde{K},L] \leq V \leq \cz{G}{V}$ so the images of $K$ and $L$
    in $\gl{}{V}$ commute.
    In particular, $\cz{V}{K}$ is $L$-invariant.
    Consider the action of $L$ on $V/\cz{V}{K}$.
    Now $\cz{V}{a} \leq \cz{V}{K}$ so \CoprimeActionQuot\ implies that $a$
    is fixed point free on $V/\cz{V}{K}$.
    Since $L = [L,a]$ and $L$ is perfect,
    Theorem~\ref{m:1} implies that $L$ is trivial on $V/\cz{V}{K}$.
    Thus \[
        [V,L] \leq \cz{V}{K}.
    \]

    Recall that $G = \widetilde{K}L$,
    so $L \normal G$.
    Then $[V,L] \normal G$ and the choice of $V$ implies $[V,L] = 1$ or $V$.
    Suppose that $[V,L] = 1$.
    Then $[L,\widetilde{K},L] \leq [V,L] = 1$ and $[\widetilde{K},L,L] = 1$
    so the Three Subgroups Lemma forces $[L,L,\widetilde{K}] = 1$.
    Then $[L,\widetilde{K}] = 1$ since $L$ is perfect.
    Suppose that $[V,L] = V$.
    Then $[V,K] = 1$.
    Again it follows from the Three Subgroups Lemma that $[K,L] = 1$.
    Then $K \leq \cz{\widetilde{K}}{L}$.
    Since $\widetilde{K}$ is $A$-quasisimple and normalizes $L$
    this forces $\cz{\widetilde{K}}{L} = \widetilde{K}$,
    whence $[\widetilde{K},L] = 1$ in this case also.
\end{proof}
\section{An application to signalizer functors}\label{a}
We being by considering an elementary abelian $r$-group
acting coprimely on a $K$-group and using Theorem~\ref{l:1}
to analyze how various local subgroups interact with each other.

\begin{Theorem}[The Local Theorem]\label{a:1}
    Let $r$ be a prime and $A$ an elementary abelian $r$-group
    that acts coprimely on the $K$-group $G$.
    For each $a \in A\nonid$ let
    \begin{align*}
        \Omega_{a} &= \set{ K \in \comp{A}{H} }%
                    { \mbox{$H$ is an $A\cz{G}{a}$-invariant subgroup of $G$} }\\
    \intertext{and}
        \Omega     &= \bigcup_{a \in A\nonid}\Omega_{a}.
    \intertext{For each $K \in \Omega$ set}
        \cstar{K}{A}    &= \left\{\begin{array}{cl}
                                    \cz{K}{A}            & \mbox{if $\cz{K}{A}$ is solvable}\\
                                    \layerr{\cz{K}{A}}   & \mbox{if $\cz{K}{A}$ is nonsolvable}.
                            \end{array}\right.
    \end{align*}
    Let $K,L \in \Omega$,
    so that $K \in \Omega_{a}$ and $L \in \Omega_{b}$
    for some $a,b \in A\nonid$.
    \begin{enumerate}
        \item[(a)]  Suppose $[K,L] \not= 1$.
                    Then there exists a unique $X$ with \[
                        \listgen{K,L} \leq X \in \compasol{G}.
                    \]
                    If $X$ is constrained then
                    $K = L \in \comp{A}{\cz{G}{\listgen{a,b}}}$.

        \item[(b)]  $\cstar{K}{A}$ is nonabelian.

        \item[(c)]  The following are equivalent:
                    \begin{enumerate}
                        \item[(i)]  $[\cstar{K}{A},\cstar{L}{A}] \not= 1$.
                        \item[(ii)] $[K,L] \not= 1$.
                        \item[(iii)]$\cstar{K}{A} = \cstar{L}{A}$.
                    \end{enumerate}

        \item[(d)]  ``Does not commute'' is an equivalence relation on $\Omega$.
    \end{enumerate}
\end{Theorem}
\begin{proof}
    (a). Theorem~\ref{l:1} implies that there exist
    unique $\widetilde{K}$ and $\widetilde{L}$ with \[
        K \leq \widetilde{K} \in \compasol{G} \mbox{ and } %
        L \leq \widetilde{L} \in \compasol{G}.
    \]
    Then $[\widetilde{K},\widetilde{L}] \not= 1$.
    Using Lemma~\ref{p:2} it follows that
    either $\widetilde{K}$ and $\widetilde{L}$ are
    both semisimple or both constrained.
    Suppose they are both semisimple.
    Since distinct $A$-components commute,
    we have $\widetilde{K} = \widetilde{L}$.
    Put $X = \widetilde{K}$.
    Hence we may assume that $\widetilde{K}$ and $\widetilde{L}$
    are both constrained.

    Theorem~\ref{l:1} implies that $[K,a] = 1$ and
    $\widetilde{K} = K\ff{\widetilde{K}}$.
    Suppose $[\widetilde{L},a]$ is nonsolvable.
    Since $\widetilde{L}$ is an $(A,\mbox{sol})$-component
    it follows that $\widetilde{L} = [\widetilde{L},a]$.
    Also $\widetilde{L} \not= \widetilde{K}$ as
    $[\widetilde{K},a] \leq \ff{\widetilde{K}}$.
    Theorem~\ref{l:1}(e) implies that $[\widetilde{K},\widetilde{L}] = 1$,
    a contradiction.
    Thus $[\widetilde{L},a] \leq \sol{\widetilde{L}}$.
    Then $[L,a] \leq \sol{\widetilde{L}} \cap L \leq \sol{L} = \zz{L}$
    and Lemma~\ref{p:5} implies $[L,a] = 1$.
    To summarize, $[K,a] = [L,a] = 1$.
    Similarly $[K,b] = [L,b] = 1$.
    Now $K \in \Omega_{a}$ and $[K,a] = 1$ so $K \in \comp{A}{\cz{G}{a}}$.
    As $[K,b] = 1$ we have $K \in \comp{A}{\cz{G}{\listgen{a,b}}}$.
    Similarly $L \in \comp{A}{\cz{G}{\listgen{a,b}}}$.
    As $[K,L] \not= 1$,
    this forces $K = L$.
    The uniqueness of $\widetilde{K}$ and $\widetilde{L}$
    forces $\widetilde{K} = \widetilde{L}$.
    Put $X = \widetilde{K}$.

    (b). Lemma~\ref{aqs:5} implies that either
    $\cz{K}{A}$ is solvable or $\layerr{\cz{K}{A}}$ is quasisimple.
    Theorem~\ref{k:4}(a) implies that $\cz{K}{A}$ is nonabelian.
    Hence the result.

    (c). Trivially (i) implies (ii).
    Suppose (ii) holds.
    Choose $X$ as in (a).
    If $X$ is constrained then $K = L$ so $\cstar{K}{A} = \cstar{L}{A}$.
    Suppose $X$ is semisimple.
    Two applications of Lemma~\ref{aqs:10} imply
    $\cstar{K}{A} = \cstar{X}{A} = \cstar{L}{A}$ so (iii) holds.
    By (b), (iii) implies (i).

    (d). Trivially, $\cstar{K}{A} = \cstar{L}{A}$
    defines an equivalence relation on $\Omega$.
\end{proof}

The reader is assumed to be familiar with elementary
Signalizer Functor Theory,
for example the notion of $\theta$-subgroups.
See \cite{F2}.
In the following result,
it is not necessary to assume $G$ to be a $K$-group.
It can be applied to study the $\theta$-subgroups in
a minimal counterexample to the Nonsolvable Signalizer Functor Theorem.

\begin{Theorem}[The Global Theorem]\label{a:2}
    Let $r$ be a prime and $A$ an elementary abelian $r$-group
    with rank at least $3$.
    Suppose that $A$ acts on the group $G$ and that
    $\theta$ is an $A$-signalizer functor on $G$.
    Assume that $\theta(a)$ is a $K$-group for all $a \in A\nonid$.
    For each $a \in A\nonid$ let
    \begin{align*}
        \Omega_{a}  &=  \{ K \in \comp{A}{H} \mid %
                        \begin{array}[t]{l}
                            \mbox{$H$ is a $\theta$-subgroup of $G$},\\
                            \theta(a) \leq H \mbox{ and }\\
                            \mbox{$H$ is a $K$-group.} \}
                        \end{array}
    \intertext{and}
        \Omega     &= \bigcup_{a \in A\nonid}\Omega_{a}.
    \intertext{For each $K \in \Omega$ set}
        \cstar{K}{A}    &= \left\{\begin{array}{cl}
                                    \cz{K}{A}            & \mbox{if $\cz{K}{A}$ is solvable}\\
                                    \layerr{\cz{K}{A}}   & \mbox{if $\cz{K}{A}$ is nonsolvable}.
                            \end{array}\right.
    \end{align*}
    Let $K,L \in \Omega$.
    The following are equivalent:
    \begin{enumerate}
        \item[(i)]  $[\cstar{K}{A},\cstar{L}{A}] \not= 1$.
        \item[(ii)] $[K,L] \not= 1$.
        \item[(iii)]$\cstar{K}{A} = \cstar{L}{A}$.
    \end{enumerate}
    In particular, ``Does not commute'' is an equivalence relation on $\Omega$.
\end{Theorem}
\begin{proof}
    Trivially (i) implies (ii).
    Also (iii) implies (i) by Theorem~\ref{a:1}(b).
    Suppose that (ii) holds.
    Lemma~\ref{aqs:12}(b),
    with $A$ in the role of $A^{*}$,
    implies there exists $B \in \hyp{A}$ with $\cz{K}{B}$ overdiagonal
    and $[\cz{K}{B},L] \not= 1$.
    Another application of Lemma~\ref{aqs:12}(b),
    with $B$ in the role of $A^{*}$,
    implies there exists $C \in \hyp{B}$ with
    $\cz{L}{C}$ overdiagonal and $[\cz{K}{B},\cz{L}{C}] \not= 1$.
    Then $[\cz{K}{C},\cz{L}{C}] \not= 1$ and
    Lemma~\ref{aqs:12}(a) implies that both
    $\cz{K}{C}$ and $\cz{L}{C}$ are $A$-quasisimple.
    Now $A$ has rank at least $3$ so $C \not= 1$
    and then $\theta(C)$ is a $K$-group.
    Set $M = \theta(C)$.

    Since $K \in \Omega_{a}$ there exists a $\theta$-subgroup $H_{a}$
    with $\theta(a) \leq H_{a}$ and $K \in \comp{A}{H_{a}}$.
    Now $K$ is a $\theta$-subgroup so $\cz{K}{C} \leq M$.
    In fact, $\cz{K}{C} \subnormal M \cap H_{a}$ since $K \subnormal H_{a}$
    so as $\cz{K}{C}$ is $A$-quasisimple,
    we have $\cz{K}{C} \in \comp{A}{M \cap H_{a}}$.
    Also, $\cz{M}{a} \leq M \cap \theta(a) \leq M \cap H_{a}$.
    Similarly, there exists a $\theta$-subgroup $H_{b}$
    with $\cz{L}{C} \in \comp{A}{M \cap H_{b}}$
    and $\cz{M}{b} \leq M \cap H_{b}$.
    The Local Theorem,
    with $M$, $\cz{K}{C}$ and $\cz{L}{C}$
    in the roles of $G$, $K$ and $L$ respectively,
    implies that $\cstar{K}{A} = \cstar{L}{A}$,
    so (iii) holds.
\end{proof}
\bibliographystyle{amsplain}

\begin{thebibliography}{99}

    \bibitem{A}     M. Aschbacher,
                    {\em Finite group theory, first edition,}
                    (Cambridge University Press, 1986).

    \bibitem{BGL}   N. Burgoyne, R.L. Griess, R. Lyons,
                    {\em Maximal subgroups and automorphisms of Chevalley Groups,}
                    Pacific J. Math. {\bf 71} (1977) 365--403

    \bibitem{C}     B.N. Cooperstein,
                    {\em Maximal subgroups of $G_{2}(2^{n})$}
                    J. Algebra {\bf 70} (1981) 23--26

    \bibitem{F1}    P. Flavell,
                    {\em Automorphisms of soluble groups,}
                    Proceedings of the London Mathematical Society 2016 112 (4): 623-650
                    doi: 10.1112/plms/pdw005


    \bibitem{F2}    P. Flavell,
                    {\em A new proof of the Solvable Signalizer Functor Theorem,}
                    J. Algebra {\bf 398} (2014) 350-363

    \bibitem{F3}    P. Flavell,
                    {\em $RC_{G}(R)$-signalizers in finite groups,}
                    J. Algebra {\bf 323} (2010) 1983--1992

    \bibitem{F4}    P. Flavell,
                    {\em A Hall-Higman-Shult type theorem for arbitrary finite groups,}
                    Invent. Math. {\bf 164} (2006) 361--397

    \bibitem{F5}    P. Flavell,
                    {\em An equivariant analogue of Glauberman's $ZJ$-Theorem,}
                    J. Algebra {\bf 257} (2002) 249--264

    \bibitem{G}     G. Glauberman,
                    {\em On solvable signalizer functors in finite groups,}
                    Proc. Lond. Math. Soc. (3) {\bf 3} (1976) 1--27

    \bibitem{GLS}   D. Gorenstein, R. Lyons, R.M. Solomon,
                    {\em The classification of the finite simple groups,}
                    (Mathematical Surverys and Monographs, {\bf 40},
                    American Math. Soc.,
                    Providence Rhode Island )

    \bibitem{GLS3}   D. Gorenstein, R. Lyons, R.M. Solomon,
                    {\em The classification of the finite simple groups, Number 3}
                    (Mathematical Surverys and Monographs, {\bf 40},
                    American Math. Soc.,
                    Providence Rhode Island

    \bibitem{I}     I.M. Isaacs,
                    {\em Finite group theory,}
                    (American Math. Soc.,
                    Providence, Rhode Island, 2008.)

    \bibitem{S}     H. Kurzweil, B. Stellmacher,
                    {\em The theory of finite groups, an introduction,}
                    (Universitext, Springer-Verlag, New York, 2004)

    \bibitem{St}    R. Steinberg,
                    {\em Lectures on Chevalley groups,}
                    (Yale University, 1968)

    \bibitem{M}     G. Malle,
                    {\em The Maximal Subgroups of $\mbox{}^{2}F_{4}(q^{2})$,}
                    J. Algebra {\bf 139} (1991) 52--69

    \bibitem{McB1}  P.P. McBride,
                    {\em Near solvable signalizer functors on finite groups,}
                    J. Algebra {\bf 78}(1) (1982) 181-214

    \bibitem{McB2}  P.P. McBride,
                    {\em Nonsolvable signalizer functors on finite groups,}
                    J. Algebra {\bf 78}(1) (1982) 215-238

    \bibitem{Suz}   M. Suzuki,
                    {\em On a class of doubly transitive groups,}
                    Ann. of Math. {\bf 75} (1962) 105-145

\end{thebibliography}

\end{document}